\documentclass[12pt]{amsart}
\usepackage{epic, eepic, amsfonts, latexsym,amssymb,graphicx,multicol,mathrsfs,color,xypic,amscd, amsxtra, verbatim, paralist, xspace, url, euscript, stmaryrd}
\usepackage{hyperref}
\usepackage[all]{xy}

\newenvironment{enumeratea} {\begin{enumerate}[\upshape (a)]} {\end{enumerate}}
\newenvironment{enumeratei} {\begin{enumerate}[\upshape (i)]} {\end{enumerate}}
\newenvironment{enumerate1} {\begin{enumerate}[\upshape (1)]} {\end{enumerate}}

\newtheorem{lem}{Lemma}[subsection]
\newtheorem{thm}[lem]{Theorem}

\newtheorem{prop}[lem]{Proposition}

\newtheorem{cor}[lem]{Corollary}

\theoremstyle{definition}
\newtheorem{defn}[lem]{Definition}
\newtheorem{example}[lem]{Example}

\newtheorem*{defn*}{Definition}
\newtheorem*{thm*}{Theorem}
\newtheorem*{mainthm*}{Main Theorem}
\newtheorem*{cor*}{Corollary}

\theoremstyle{remark}
\newtheorem{remark}[lem]{Remark}

\numberwithin{equation}{subsection}

\newcommand{\epf}{\qed \vspace{+10pt}}
\newcommand\cA{\mathcal{A}} \newcommand\cB{\mathcal{B}} \newcommand\cC{\mathcal{C}}\renewcommand\cD{\mathcal{D}} 
\newcommand\cF{\mathcal{F}} \newcommand\cG{\mathcal{G}}
\newcommand\cI{\mathcal{I}}\newcommand\cJ{\mathcal{J}}\newcommand\cK{\mathcal{K}}\renewcommand\cL{\mathcal{L}}\newcommand\cM{\mathcal{M}}
\newcommand\cO{\mathcal{O}}\newcommand\cP{\mathcal{P}}
\newcommand\cV{\mathcal{V}}\newcommand\cW{\mathcal{W}}
\newcommand\cX{\mathcal{X}}\newcommand\cY{\mathcal{Y}}\newcommand\cZ{\mathcal{Z}}

\renewcommand\AA{\mathbb{A}}\newcommand\FF{\mathbb{F}}\newcommand\GG{\mathbb{G}}\newcommand\PP{\mathbb{P}}\newcommand\QQ{\mathbb{Q}}
\newcommand\ZZ{\mathbb{Z}}

\newcommand\fm{\mathfrak{m}}

\newcommand\fp{\mathfrak{p}}

\newcommand\fq{\mathfrak{q}}

\newcommand\id{\mathrm{id}}

\newcommand\spec{\operatorname{Spec}}

\newcommand\hookarr{\hookrightarrow}

\newcommand\im{\operatorname{im}}

\renewcommand{\setminus}{\smallsetminus}

\newcommand{\Le}{\textup{Lis-{\'e}t}}

\def\gitq{/\hspace{-0.1cm}/}

\newcommand{\simm}{\kern-4pt\sim \, \,}
\newcommand{\simmc}{\kern-4pt\sim_c \, \,}
\newcommand{\red}{_{\mathrm{red}}}

\renewcommand{\ss}{\operatorname{ss}}
\newcommand{\s}{\operatorname{s}}
\newcommand{\rrarrows}{\rightrightarrows}
\newcommand{\coker}{\operatorname{coker}}

\newcommand{\Eq}{\operatorname{Eq}}
\newcommand{\Hom}{\operatorname{Hom}}

\newcommand{\sSpec}{\operatorname{\mathcal{S}pec}}

\newcommand{\Proj}{\operatorname{Proj}}

\newcommand{\Sym}{\operatorname{Sym}}

\newcommand{\oh}{\cO}

\newcommand{\Spec}{\spec}

\newcommand{\tensor} {\otimes}

\newcommand{\Sch}{\operatorname{Sch}}

\newcommand{\Sets}{\operatorname{Sets}}

\newcommand{\iso}{\stackrel{\sim}{\to}}

\newcommand{\dlim}{{\displaystyle \lim_{\longrightarrow}}\,}

\newcommand{\GL}{\operatorname{GL}}
\newcommand{\PGL}{\operatorname{PGL}}
\newcommand{\SL}{\operatorname{SL}}

\newcommand{\mapsonto} {\twoheadrightarrow}
\renewcommand{\bar}{\overline}

\renewcommand{\tilde} {\widetilde}
\newcommand{\dual}{\vee}

\makeatletter
\def\blfootnote{\xdef\@thefnmark{}\@footnotetext}
\makeatother
\begin{document}

\title[Adequate moduli spaces]{Adequate moduli spaces and \\geometrically reductive group schemes} 
\author[Alper]{Jarod Alper}

\address[Alper]{Mathematical Sciences Institute\\
Australian National University\\
Canberra, ACT 0200, Australia}
\email{jarod.alper@anu.edu.au}

\date{May 13, 2010}

\maketitle

\begin{abstract}
We introduce the notion of an adequate moduli space.  The theory of adequate moduli spaces provides a framework for studying algebraic spaces which geometrically approximate algebraic stacks with reductive stabilizers in characteristic $p$.  The definition of an adequate moduli space generalizes the existing notion of a good moduli space to characteristic $p$ (and mixed characteristic).   The most important examples of an adequate moduli space are: (1) the morphism from the quotient stack $[X^{\ss}/G]$ of the semistable locus to the GIT quotient $X^{\ss} \gitq G$ and (2) the morphism from an algebraic stack with finite inertia to the 
Keel$-$Mori coarse moduli space.  It is shown that most of the fundamental properties of the GIT quotient $X^{\ss} \gitq G$ follow from only the defining properties of an adequate moduli space.
We provide applications of adequate moduli spaces to the structure of geometrically reductive and reductive group schemes.  In particular, results of Seshadri and Waterhouse are generalized.
The theory of adequate moduli spaces provides the possibility for intrinsic constructions of projective moduli spaces in characteristic $p$.
\end{abstract}

\setcounter{tocdepth}{1}
\tableofcontents

\section{Introduction}

\subsection*{Background and motivation}
In characteristic $0$, any representation of a finite group $G$ is completely reducible.  Therefore, the functor from $G$-representations to vector spaces $V \mapsto V^G$ given by taking invariants is exact.  In particular, if $G$ acts on an affine scheme $X$ and $Z \subseteq X$ is an invariant closed subscheme, every $G$-invariant function on $Z$ lifts to a $G$-invariant function on $X$.  In fact, for any algebraic group $G$, these properties are equivalent and give rise to the notion of a \emph{linearly reductive group}.

\medskip \noindent
In characteristic $p$, if $p$ divides the order $|G| = N$ of a finite group $G$, then the above properties can fail.  However, if $f$ is a $G$-invariant function on an invariant closed subscheme $Z$ of an affine scheme $X$ and $\tilde f$ is any (possibly non-invariant) lift to $X$, then $\prod_{g \in G} g \cdot \tilde f$ is a $G$-invariant function on $X$ which is a lift of $f^N$.  This motivates the definition of \emph{geometric reductivity} for an algebraic group $G$: for every action of $G$ on an affine scheme $X$, every invariant closed subscheme $Z \subseteq X$ and every $f \in \Gamma(Z, \oh_Z)^G$, there exist an integer $n >0$ and $g \in \Gamma(X, \oh_X)^G$ extending $f^n$.

\medskip \noindent
In positive characteristic, linearly reductive groups are rare (as the connected component is always a torus) while many algebraic groups (for example, $\GL_n$, $\SL_n$, $\PGL_n$) are geometrically reductive.  The notion of geometric reductivity of an algebraic group $G$ was introduced by Mumford in the preface of \cite{git}.  Nagata showed in \cite{nagata_invariants-affine}  that if a geometrically reductive group $G$ acts on a finite type affine scheme $\Spec(A)$ over a field $k$, then $A^G$ is finitely generated over $k$ and $\Spec(A^G)$ is a suitably nice quotient.  Mumford conjectured that the notions of geometric reductivity and reductivity were equivalent for an algebraic group; this result was proved by Haboush \cite{haboush}. Therefore, \emph{geometric invariant theory} (GIT) for reductive group actions could be developed in positive characteristic (see \cite[Appendix 1.C]{git3}) which in turn was employed with great success to construct various moduli spaces in characteristic $p$.  Since the Hilbert-Mumford criterion holds in positive characteristic (see  \cite[Appendix 2.A]{git3}), most arguments using GIT to construct moduli spaces in characteristic $0$ extend to positive characteristic.  For instance, one can use GIT to construct moduli spaces of bundles and sheaves over projective varieties in positive characteristic; \cite{maruyama_sheaves},  \cite{gieseker_surface_bundles},  and \cite{seshadri_bundles}.

\medskip \noindent
If $G$ is a geometrically reductive group acting on an affine scheme $X=\Spec(A)$ over a field $k$, then we can consider the quotient stack $\cX = [X / G]$.  There is a natural map
$$\phi: \cX \to Y:=\Spec(A^G)$$
which is easily seen to have the following properties:
\begin{enumerate1}
\item For every surjection of quasi-coherent $\oh_{\cX}$-algebras $\cA \to \cB$ and section $t \in \Gamma(\cX, \cB)$ there exist an integer $N > 0$ and a section $s \in \Gamma(\cX, \cA)$ such that $s \mapsto t^N$.
\item The homomorphism $\Gamma(Y, \oh_Y) \to \Gamma(\cX, \oh_{\cX})$ is an isomorphism.
\end{enumerate1}

\medskip \noindent
These properties motivate the following definition: for any algebraic stack $\cX$, we say that a morphism $\phi: \cX \to Y$ to an affine scheme is an \emph{adequate moduli space} if properties (1) and (2) are satisfied.  Because $Y$ is assumed to be affine, property (1) is  intrinsic to $\cX$ and independent of the morphism $\phi$.  When $Y$ is not affine, one has to consider the local versions of properties (1) and (2); see the following definition.

\medskip \noindent
The purpose of this paper is to develop the theory of adequate moduli spaces and then consider applications to the structure of geometrically reductive group schemes over an arbitrary base.

\subsection*{The definition and main properties}  The main definition of this paper is the following.

\begin{defn*} A quasi-compact and quasi-separated morphism $\phi: \cX \to Y$ from an algebraic stack to an algebraic space is an \emph{adequate moduli space} if the following two properties are satisfied:
\begin{enumerate1}
\item For every surjection of quasi-coherent $\oh_{\cX}$-algebras $\cA \to \cB$, \'etale morphism $p: U=\Spec(A) \to Y$ and section $t \in \Gamma(U, p^* \phi_* \cB)$ there exist $N > 0$ and a section $s \in \Gamma(U, p^* \phi_*\cA)$ such that $s \mapsto t^N$.
\item The morphism $\oh_Y \to \phi_* \oh_{\cX}$ is an isomorphism.
\end{enumerate1}
\end{defn*}

\medskip \noindent
Before going further, we state now the two main examples of an adequate moduli space that the reader should keep in mind.  First, if $G$ is a reductive group over a field $k$ acting on an affine $k$-scheme $\Spec(A)$, then $[\Spec(A) / G] \to \Spec(A^G)$ is an adequate moduli space (see Theorem \ref{theorem-git}).  More generally, if $G$ acts on a projective scheme $X$ and $L$ is an ample $G$-linearization, then the quotient of the semistable locus 
$$[X^{\ss}/G] \to X^{\ss} \gitq G :=\Proj(\bigoplus_{d \ge 0} \Gamma(X, L^{\tensor d})^G)$$
 is an adequate moduli space.  
It turns out that many of the standard properties of the GIT quotient $X^{\ss} \to X^{\ss} \gitq G$ can be seen to follow directly from properties (1) and (2); see the Main Theorem below.
Second, for an algebraic stack $\cX$ with finite inertia, the morphism $\phi: \cX \to Y$ to the Keel$-$Mori coarse moduli space is an adequate moduli space (see Theorem \ref{theorem-coarse}).

%  It turns out that properties (1) and (2) capture the stack-intrinsic properties of such GIT quotients stacks $[X/G]$ and that these properties alone suffice to show that the quotient $Y$ inherits nice geometric properties.

\medskip \noindent

\medskip \noindent One of the key insights in this paper is a generalization of Serre's criterion providing a characterization of affineness of an algebraic space in terms of the existence of liftings of powers of  sections along a surjective morphism of sheaves of quasi-coherent algebras.  More specifically, we define a morphism to be  \emph{adequately affine} if property (1) in the above definition is satisfied (see Definition \ref{definition-adequately-affine}) and we provide various equivalent formulations (see Lemmas \ref{lemma-adequately-affine-equiv} and \ref{lemma-affine-base-equiv}).  We prove the following generalization of Serre's criterion: if $f: X \to Y$ is a quasi-compact and quasi-separated morphism of algebraic spaces, then $f$ is adequately affine if and only if $f$ is affine (see Theorem \ref{serres-criterion}).

\medskip \noindent The notion of a \emph{good moduli space} (see \cite{alper-good}) is defined by replacing property (1) with the requirement that the push-forward functor $\phi_*$ be exact on quasi-coherent sheaves (that is, $\phi$ is cohomologically affine).  Any good moduli space is certainly an adequate moduli space; the converse is true in characteristic 0 (see Proposition \ref{proposition-good}).

\medskip \noindent
Section \ref{section-adequate} is devoted to characterizing ring maps $A \to B$ with the property that for all $b \in B$, there exist $N > 0$ and $a \in A$ such that $a \mapsto b^N$; such ring maps are called \emph{adequate} (see Definition \ref{definition-adequate}) and play an essential role in this paper. This notion is not stable under base change so we introduce \emph{universally adequate} ring maps (see Definition \ref{definition-universally-adequate}).  The key fact here is that if
$A \to B$ is finitely presented, then $A \to B$ is universally adequate with locally nilpotent kernel if and only if $\Spec(B) \to \Spec(A)$ is an integral universal homeomorphism which is an isomorphism in characteristic $0$ (see Proposition \ref{proposition-adequate-equiv}); we refer to this notion as an \emph{adequate homeomorphism}  (see Definition \ref{definition-adequate-homeomorphism}).
In the case of actions by finite groups (or more generally finite group schemes), the integer $N$ in the definition above can be chosen to be the size of the group.  However, for non-finite geometrically reductive groups (for example, $\SL_2$), Example \ref{example-johan} shows that the integer $N$ cannot be chosen universally over all quasi-coherent $\oh_{\cX}$-algebras. 

\medskip \noindent
The following theorem summarizes the main geometric properties of adequate moduli spaces:

\begin{mainthm*}  Let $\phi: \cX \to Y$ be an adequate moduli space.  Then
\begin{enumerate1}
\item The morphism $\phi$ is surjective, universally closed and universally submersive (Theorem \ref{theorem-adequate} (\ref{surjective}, \ref{closed} and \ref{submersive})).
\item Two geometric points $x_1$ and $x_2 \in \cX(k)$ are identified in $Y$ if and only if their closures $\overline{\{x_1\}}$ and $\overline{\{x_2\}}$ in $\cX \times_{\ZZ} k$ intersect  (Theorem \ref{theorem-adequate} (\ref{intersection})).
\item If $Y' \to Y$ is any morphism of algebraic spaces, then $\cX \times_Y Y' \to Y'$ factors as an adequate moduli space $\cX \times_Y Y' \to \tilde{Y}$ followed by an adequate homeomorphism $\tilde Y \to Y'$ (Proposition \ref{proposition-base-change}).
\item Suppose that $\cX$ is of finite type over a Noetherian scheme $S$.  Then $Y$ is of finite type over $S$ and for every coherent $\oh_{\cX}$-module $\cF$, $\phi_* \cF$ is coherent.  (Theorem \ref{theorem-finiteness}).
\item The morphism $\phi$ is universal for maps to algebraic spaces which are either locally separated or Zariski-locally have affine diagonal (Theorem \ref{theorem-uniqueness}).
\item Adequate moduli spaces are stable under flat base change and descend under morphisms $Y' \to Y$ which are fpqc (that is, faithfully flat and every quasi-compact open subset of $Y$ is the image of a quasi-compact open subset of $Y'$) (Proposition \ref{proposition-base-change}). 
\end{enumerate1}
\end{mainthm*}

\medskip \noindent
Part (4) above can be considered as a generalization of Hilbert's 14th problem and the statement that if $G$ is a reductive group over $k$ and $A$ is a finitely generated $k$-algebra, then $A^G$ is finitely generated over $k$ (see \cite{nagata_invariants-affine} or \cite[Appendix 1.C]{git3}). It also generalizes \cite[Theorem 4.16(xi)]{alper-good}) and Seshadri's result \cite[Theorem 2]{seshadri_reductivity}.  See the discussion in Section \ref{section-finiteness}.

\medskip \noindent
Part (5) implies that adequate moduli spaces are unique in a certain subcategory of algebraic spaces with a mild separation hypothesis.  This result implies that GIT quotients by reductive groups over a field are also unique in this subcategory of algebraic spaces.

\medskip \noindent
We also prove the following characterization of algebraic stacks admitting Keel$-$Mori coarse moduli spaces (Theorem \ref{theorem-coarse}):
\begin{thm*} 
 If $\cX$ is an algebraic stack with quasi-finite and separated diagonal, the following are equivalent:
\begin{enumerate1}
\item The inertia $I_{\cX} \to \cX$ is finite.
\item There exists a coarse moduli space $\phi: \cX \to Y$ with $\phi$ separated.
\item There exists an adequate moduli space $\phi: \cX \to Y$.
\end{enumerate1}
\end{thm*}

\subsection*{Applications to geometrically reductive group schemes}  The theory of adequate moduli spaces allows for several interesting applications to the structure of geometrically reductive and reductive group schemes. 
Building off the work of Seshadri in \cite{seshadri_reductivity}, we first systematically develop the theory of geometrically reductive group schemes in Section \ref{section-adequate-groups} and we then deduce the foundational properties of quotients by geometrically reductive group schemes (see Theorem \ref{theorem-git}).  Our approach differs from Seshadri's  \cite{seshadri_reductivity} where the main interest is only studying quotients by reductive group schemes.  We can also consider group schemes which may not be smooth, affine or have connected fibers. 

\begin{defn*} Let $S$ be an algebraic space. A flat, finitely presented, separated group algebraic space $G \to S$ is \emph{geometrically reductive} if $BG \to S$ is an adequate moduli space.
\end{defn*}

\medskip \noindent
If $S= \Spec(R)$, then $G \to \Spec(R)$ is geometrically reductive if for every surjection $A \to B$ of $G$-$R$-algebras and $b \in B^G$, there exist $N > 0$ and $a \in A^G$ such that $a \mapsto b^N$.  The notion of geometric reductivity can be formulated in various ways (see Lemmas \ref{lemma-group-equiv} and \ref{lemma-group-equiv-affine}).  When $G \to \Spec(R)$ is smooth with $R$ Noetherian and satisfies the resolution property,  this definition is equivalent to Seshadri's notion (see \cite[Theorem 1]{seshadri_reductivity} and Remark \ref{remark-seshadri}).  Furthermore, Seshadri's generalization of Haboush's theorem can be extended as follows (see Theorem \ref{theorem-reductive}).

 \begin{thm*}  
 Let $G \to S$ be a smooth group scheme.  Then $G \to S$ is geometrically reductive if and only if the geometric fibers are reductive and $G/G^{\circ} \to S$ is finite.
 \end{thm*}

\medskip \noindent 
 Generalizing the main result of \cite{waterhouse}, we prove the following result (see Theorem \ref{theorem-finite-group}).

\begin{thm*} 
Let $G \to S$ be a quasi-finite, separated, flat group algebraic space.  Then $G \to S$ is geometrically reductive if and only if $G \to S$ is finite.
\end{thm*}

\medskip \noindent
We offer the following generalization of Matsushima's theorem (see Section \ref{section-matsushima} for a historical discussion, and Theorem \ref{theorem-matsushima} and Corollary \ref{corollary-matsushima} for a proof).

\begin{thm*} 
Let $G \to S$ be a geometrically reductive group algebraic space and $H \subseteq G$ a flat, finitely presented and separated subgroup algebraic space.  If $G/H \to S$ is affine, then $H \to S$ is geometrically reductive.  If $G \to S$ is affine, the converse is true.  In particular, if $G \to S$ is a reductive group scheme and $H \subseteq G$ a flat, finitely presented and separated subgroup scheme, then $H \to S$ is reductive if and only if $G/H \to S$ is affine.
\end{thm*}

\subsection*{Potential applications} 
The theory of good moduli spaces (which is the analogous notion in characteristic $0$) has already had several interesting applications to the log minimal model program for $\bar{M}_g$ (see \cite{afs} and \cite{afsw}).  The theory of adequate moduli spaces is likely indispensable in extending these results to characteristic $p$.  In fact, adequate moduli spaces have already made a prominent appearance in \cite{melo-viviani} and \cite{bfmv}.

\medskip \noindent
One might hope that the theory of adequate moduli spaces allows for intrinsic constructions of proper (or projective) moduli spaces in characteristic $p$.  The general strategy is:
\begin{enumerate}
\item Show that the moduli problem is represented by an algebraic stack $\cX$.
\item Use geometric properties of the moduli problem to show that there exists an adequate moduli space $\cX \to Y$ where $Y$ is an algebraic space.
\item Use a valuative criterion on $\cX$ to show that $Y$ is proper.  To show that $Y$ is projective, show that a certain tautological line bundle on $\cX$ descends to $Y$ and then use intersection theory techniques to show that the descended line bundle is ample.
\end{enumerate}

\medskip \noindent
Step (1) can often be accomplished by verifying deformation-theoretic properties 
of the moduli problem \cite{artin_versal}.  It can be the case that $\cX$ is not a global quotient stack such as for the moduli stack of semistable curves (see \cite{kresch-flattening}).  Moreover, it is often the case that $\cX$ is not known to be a global quotient stack such as for the moduli stacks parameterizing Bridgeland semistable objects (see \cite{abramovich-polishchuk} and \cite{toda}).  Therefore, to construct an adequate moduli space, one cannot rely on the machinery of GIT.  The construction of the adequate moduli space in step (2) is often the most challenging ingredient in this procedure and can be viewed as a generalization of the Keel$-$Mori theorem \cite{keel-mori} which guarantees the existence of coarse moduli spaces for algebraic stacks with finite inertia.  The verification of properness in step (3) involves showing that the moduli stack satisfies a weak valuative criterion analogous to Langton's theorem \cite{langton} for the moduli stack of torsion-free sheaves on a smooth projective variety or \cite[Theorem 4.1.1]{abramovich-polishchuk} for the moduli stack of Bridgeland semistable objects.  The strategy to establish projectivity in step (3) is analogous to Koll\'ar's proof of the projectivity of $\bar{M}_g$ \cite{kollar_projectivity}. 
Although this three-step procedure is ambitious, the analogous strategy has been successfully employed in characteristic $0$ in \cite{afsw} to construct the second flip of $\bar{M}_g$.  

\subsection*{Acknowledgments}   I am indebted to Johan de Jong for providing the motivation to pursue this project and for offering many useful suggestions.  I also thank David Rydh and Ravi Vakil for helpful discussions.
%%Thank NSF!!

%%%%%%%%%%%%%%%%%%%%%%%%%%%%%%%%%%%%%%%%%%%%%%%%
\section{Conventions}

\medskip \noindent
We use the terms algebraic stack and algebraic space in the sense of \cite{lmb}.  In particular, all algebraic stacks and algebraic spaces have a quasi-compact and separated diagonal (although we sometimes superfluously state this hypothesis).  If $\cX$ is an algebraic stack, the \emph{lisse-\'etale site} of $\cX$, denoted $\Le(\cX)$, is the site where objects are smooth morphisms $U \to \cX$ from schemes $U$, morphisms are arbitrary $\cX$-morphism, and covering families are \'etale.

\subsection{$G$-$R$-modules and algebras}
Let $G \to S=\Spec(R)$ be a flat, finitely presented and separated group scheme.  Let $\epsilon: \Gamma(G, \oh_G) \to R$, $\iota: \Gamma(G, \oh_G) \to \Gamma(G, \oh_G)$ and $\delta: \Gamma(G, \oh_G) \to \Gamma(G, \oh_G) \tensor_R \Gamma(G, \oh_G)$ be the counit, coinverse and comultiplication, respectively.
A \emph{(left) $G$-$R$-module} is an $R$-module $M$ with a coaction $\sigma_M: M \to \Gamma(G, \oh_G) \tensor_R M$ satisfying the commutative diagrams:
$$
\xymatrix{
M \ar[r]^{\sigma_M} \ar[d]^{\sigma_M}	& \Gamma(G, \oh_G) \tensor_R M \ar[d]^{\id \tensor \sigma_M} \\
\Gamma(G, \oh_G) \tensor_R M \ar[r]^{\delta \tensor \id \qquad}	& \Gamma(G, \oh_G) \tensor_R \Gamma(G, \oh_G) \tensor_R M
}
\qquad
\xymatrix{
M \ar[r]^{\sigma_M \qquad } \ar[rd]^{\id}	& \Gamma(G, \oh_G) \tensor_R M \ar[d]^{\epsilon \tensor \id} \\
		& M
}
$$
A \emph{morphism of $G$-$R$-modules} is a morphism of $R$-modules $\alpha: A \to B$ such that $(\id \tensor \alpha) \circ \sigma_M = \sigma_N \circ \alpha$.  The operators of direct sum and tensor products extend to $G$-$R$-modules.  A \emph{(left) $G$-$R$-algebra} is a $G$-$R$-module $A$ with the structure of an $R$-algebra such that $R \to A$ (where $R$ has the trivial $G$-$R$-module structure) and multiplication $A \tensor_R A \to A$ are morphisms of $G$-$R$-modules.  A \emph{morphism of $G$-$R$-algebras} is a morphism of $G$-$R$-modules $\alpha: A \to B$ which is also a morphism of $R$-algebras.

\medskip \noindent
Let $BG = [S/G]$ be the classifying stack of $G \to S$.  The category of $G$-$R$-modules (of finite type) is equivalent to the category of quasi-coherent sheaves (of finite type, respectively) on $BG$.  The category of $G$-$R$-algebras (of finite type) is equivalent to the category of quasi-coherent $\oh_{BG}$-algebras (of finite type, respectively). (One defines a $G$-$R$-module or $G$-$R$-algebra to be \emph{finite type} if the underlying $R$-module or $R$-algebra, respectively, is of finite type.

\subsection{Locally nilpotent ideals}
Recall that an ideal $I$ of a ring $R$ is \emph{locally nilpotent} if for every $x \in I$ there exists $N > 0$ such that $x^N = 0$.   Of course, if $I$ is finitely generated, this is equivalent to requiring the existence of $N > 0$ such that $I^N = 0$.  An ideal $\cI \subseteq \cA$ of a quasi-coherent $\oh_{\cX}$-algebra $\cA$ is \emph{locally nilpotent} if for every object $(U \to \cX) \in \Le(\cX)$ and section $x \in \cI(U \to \cX)$, there exists $N > 0$ such that $x^N = 0$.

\subsection{Symmetric products}

If $\cX$ is an algebraic stack and $\cF$ is a quasi-coherent or finite type $\oh_{\cX}$-module, then the symmetric algebra $\Sym^* \cF$ is a quasi-coherent $\oh_{\cX}$-algebra or a finite type $\oh_{\cX}$-algebra, respectively.  This construction is functorial:  a morphism of quasi-coherent $\oh_{\cX}$-modules $\cF \to \cG$ induces a morphism of quasi-coherent $\oh_{\cX}$-algebras $\Sym^* \cF \to \Sym^* \cG$.  Note that if $\cM \subseteq \cA$ is sub-$\oh_{\cX}$-module of a quasi-coherent $\oh_{\cX}$-algebra $\cF$, then there is an induced morphism $\Sym^* \cM \to \cA$ of quasi-coherent $\oh_{\cX}$-algebras.

\begin{lem}  \label{lemma-qcoh-limit}
If $\cX$ is a Noetherian algebraic stack, then every quasi-coherent $\oh_{\cX}$-algebra is a filtered inductive limit of finite type sub-$\oh_{\cX}$-algebras.  If $\cA$ is a finite type  $\oh_{\cX}$-algebra, then there exist a coherent sub-$\oh_{\cX}$-module $\cM \subseteq \cA$ such that $\Sym^* \cM \to \cA$ is surjective.

\end{lem}

\begin{proof}  
This follows formally from \cite[15.4]{lmb} as in \cite[I.9.6.6]{ega}.  Namely, \cite[15.4]{lmb} implies that any quasi-coherent $\oh_{\cX}$-algebra is a filtered inductive limit of coherent sub-$\oh_{\cX}$-modules and each sub-$\oh_{\cX}$-module generates a finite type sub-$\oh_{\cX}$-algebra. 
\end{proof}

%%%%%%%%%%
\section{Adequacy for rings} \label{section-adequate}

\subsection{Adequate ring homomorphisms}
\begin{defn} \label{definition-adequate}
A homomorphism of rings $A \to B$ is \emph{adequate} if for every element $b \in B$, there exists an integer $N > 0$ and $a \in A$ such that $a \mapsto b^N$.
\end{defn}

\medskip \noindent
It is clear that the composition of adequate ring maps is again adequate.  

\begin{lem} \label{lemma-adequate-basic}
Let $\phi: A \to B$ be an adequate homomorphism.  Then
\begin{enumerate1}
\item If $S \subseteq A$ is a multiplicative set, then $S^{-1} A \to S^{-1}A \tensor_A B$ is adequate.
\item If $I \subseteq A$ is an ideal, then $A/I \to B/IB$ is adequate.
\item For every prime $\fp \subseteq A$, the homomorphism $A_{\fp} \hookarr A_{\fp} \tensor_A B$ and $k(\fp) \hookarr  k(\fp) \tensor_A B$ are adequate.
\item For every $\fq \subseteq B$ with $\fp = \phi^{-1}(\fq)$, the homomorphisms $A_{\fp} \to B_{\fq}$ and $k(\fp) \to k(\fq)$ are adequate.
\item If $A$ is local with maximal ideal $\fm_A$, then $B$ is local with maximal ideal $\sqrt{ \fm_A B}$.
\end{enumerate1}
\end{lem}

\begin{proof}
Let $b/g \in S^{-1} B$. For some integer $N>0 $, we have $b^N \in A$ and therefore $(b/g)^N \in S^{-1}A$.  Statements $(2) - (4)$ are clear.  For statement $(5)$, for $b \notin \sqrt{\fm_A B}$, there exist $N > 0$ and $a \in A$ with $a \mapsto b^N$ but then $a \notin \fm_A$ so $b$ is a unit.
\end{proof}

\begin{lem} \label{lemma-adequate-descent}
Let $A \to B$ be a ring homomorphism and let $A \to A'$ be a faithfully flat ring homomorphism.  If $A' \to A' \tensor_A B$ is adequate, then so is $A \to B$.
\end{lem}

\begin{proof}
We may assume that $A \to B$ is injective.  Since $A \to A'$ is faithfully flat, there is a commutative diagram
$$\xymatrix{
A \ar[r] \ar[d]	&  A' \ar@<0.5ex>[r] \ar@<-0.5ex>[r] \ar[d]		&   A' \tensor_A A' \ar[d] \\
B \ar[r] 		& B \tensor_A A' \ar@<0.5ex>[r] \ar@<-0.5ex>[r]  	&B \tensor _A A' \tensor_A A' 
}$$
where the rows are exact.  If $b \in B$, there exist $a' \in A'$ and $N > 0$ such that $a' \mapsto b^N \tensor 1$.  Since the elements $a' \tensor 1$ and $1 \tensor a'$ are equal in $A' \tensor_A A'$, we have $a' \in A$ and $a' \mapsto b^N$ in $B$.
\end{proof}

\begin{lem} \label{lemma-adequate-homeomorphism}
Let $A \hookarr B$ be an adequate inclusion of rings.  Then $\Spec(B) \to \Spec(A)$ is an integral homeomorphism. 
\end{lem}

\begin{proof} It is clear that $A \to B$ is integral.   By Lemma \ref{lemma-adequate-basic}, for every $\fp \subseteq A$, the fiber $k(\fp) \to k(\fp) \tensor_A B$ is adequate which implies that $(k(\fp) \tensor_A B)_{\text{red}}$ is a field.  Since $\Spec(B) \to \Spec(A)$ is integral, injective and dominant, it is a homeomorphism.
\end{proof}.

\begin{lem} \label{lemma-adequate-char0}
Let $A \hookarr B$ be an adequate inclusion of $\QQ$-algebras.  Then $A = B$.
\end{lem}

\begin{proof}  An element $b \in B$ determines a ring homomorphism $\pi: \QQ[x] \to B$ and $\pi^{-1}(A) \hookarr \QQ[x]$ is adequate.  It thus suffices to handle the case when $A \subseteq B = \QQ[x]$.  There exists an $n > 0$ such that $\QQ[x^n] \subseteq A \subseteq \QQ[x]$ so that $A \to B=\QQ[x]$ is finite and $A$ is necessarily Noetherian.   For a maximal ideal $\fq \subseteq B$ with $\fp = \fq \cap A$, the map $A_{\fp} \hookarr B_{\fq}$ is adequate, where $B_{\fq}$ is a discrete valuation ring.  If $I = \ker(A_{\fp}[t] \to B_{\fq})$ where $t \mapsto x$, then for some $N > 0$ and $a \in A_{\fp}$, we have $(t+1)^N -a \in I$.  It follows that $\Omega_{B_{\fq} / A_{\fp}} = 0$ as $N (t+1)^{N-1} dt = 0$ and $t+1 \in B_{\fq}$ is a unit.  Therefore, $\Spec(B) \to \Spec(A)$ is finite and \'etale.  By Lemma \ref{lemma-adequate-homeomorphism}, it is also a homeomorphism and therefore an isomorphism. 
\end{proof}

\begin{lem} \label{lemma-transcendental}
Let $k \hookarr k'$ be an adequate inclusion of fields of characteristic $p$.  Suppose that $k$ is transcendental over $\FF_p$.  Then $k \hookarr k'$ is purely inseparable.
\end{lem}

\begin{proof}
There is a factorization $k \subseteq k_0 \subseteq k'$ such that $k \subseteq k_0$ is separable and $k_0 \subseteq k'$ is purely inseparable.  Since $k \hookarr k_0$ is adequate, it suffices to show that for every adequate and separable field extension $k \hookarr k'$ with $k$ transcendental over $\FF_p$,  we have $k = k'$.

\medskip \noindent
We may assume that $k' = k(\alpha)$ with $\alpha^q = a \in k$ where $q \neq p$ is a prime and $a \in k$ is transcendental over $\FF_p$.  Suppose that $| k(\alpha) : k | = q$ so that $1, \alpha, \ldots, \alpha^{q-1}$ form a basis of $k'$ over $k$.  There exists $N > 0$ such that $(\alpha + a)^N = b \in k$.  Write $N = p^k N'$ with $p \nmid N'$.  Then $(\alpha + a)^N = (\alpha^{p^k} + a^{p^k})^{N'} \in k$, $k(\alpha^{p^k}) = k(\alpha)$ and $(\alpha^{p^k})^q = a^{p^k} \in k$ is transcendental.  So we may assume $p \nmid N$.  We can write
$$\begin{aligned}
 (\alpha + a)^N & = \sum_{i=0}^N{N \choose i} \alpha^i a^{N-i} 
 = \sum_{j=0}^{q-1} \bigg( \sum_{i=0}^{\lfloor (N-j)/q \rfloor} {N \choose j+qi } a^{N-j-qi+i} \bigg) \alpha^j 
\end{aligned}$$
By looking at the coefficient of $\alpha$, since $p$ does not divide $N$, we obtain a monic relation for $a$ over $\FF_p$, which is a contradiction.  Suppose that $| k(\alpha) : k | < q$ so that $x^q - a$ is reducible over $k$.  Then $a = b^q$ for some $b \in k$ (see \cite[Theorem 9.1]{lang_algebra}). Then $\alpha b^{-1} = \xi_q$ is a $q$th root of unity and $k'= k(\xi_q)$.  Let $t \in k$ be a transcendental element.  There exists an integer $N > 0$ such that $(t+\xi_q)^N = b \in k$.  We may assume that $p \nmid N$.  In the expansion of $(t+\xi_q)^N$ in terms of the basis $1, \xi_q, \ldots, \xi_q^{q-2}$ of $k'$ over $k$, the coefficient of $\xi_q$ is a polynomial $g(t) = N t^{N-1} + \cdots + (\text{lower degree terms})\in \FF_p[t]$ which must be 0.  This contradicts the fact that $t \in k$ is transcendental.
\end{proof}

\begin{remark} \label{remark-counter}
The hypothesis that $k$ be transcendental over $\FF_p$ is necessary.  An inclusion of finite fields $\FF_q \hookarr \FF_{q^n}$ is adequate as every element satisfies $x^{q^n-1} = 1$.  In fact, if $k$ is algebraic over $\FF_p$, then $k \hookarr \overline{\FF}_p$ is adequate.
\end{remark}
%%%%%%%%%%%%%%%%%%%%%%%%%

\subsection{Universally adequate ring homomorphisms}
If $A \to B$ is an adequate inclusion of rings, the base change $A' \to A' \tensor_A B$ by an $A$-algebra $A'$ is not necessarily adequate; similarly, $\Spec(B) \to \Spec(A)$ is a homeomorphism (see Lemma \ref{lemma-adequate-homeomorphism}) but is not necessarily a universal homeomorphism.  For instance, 
$$\FF_q \hookarr \FF_{q^n}$$
 is adequate but $\FF_{q^n} \to \FF_{q^n} \tensor_{\FF_q} \FF_{q^n} \cong \times_{i=1}^n \FF_{q^n}$ is not adequate.  Furthermore, if $B$ is any $\FF_p$-algebra and $\fm \subseteq B$ is a maximal ideal with residue field $\FF_{p^n}$ with $n > 1$, then let $A = \pi^{-1}(\FF_p)$ where $\pi: B \to \FF_{p^n}$.  Then $A \subseteq B$ is adequate but this is not stable under base change.  
This discussion motivates the following definition:

\begin{defn} \label{definition-universally-adequate}
A ring homomorphism $A \to B$ is \emph{universally adequate} if for every $A$-algebra $A'$ the ring homomorphism $A' \to A' \tensor_A B$ is adequate.
\end{defn}

\begin{remark} \label{remark-smooth}
A ring homomorphism $A \to B$ is universally adequate if and only if for every $n$ the ring homomorphism $A[x_1, \ldots, x_n] \to B[x_1, \ldots, x_n]$ is adequate.  Indeed, any $b' \in A' \otimes_A B$ is the image of an element in $B[x_1, \ldots, x_n]$ for some $n$.   Furthermore, by Lemma \ref{lemma-adequate-descent}, the property of being universally adequate descends under faithfully flat ring homomorphisms.
\end{remark}

\begin{lem} \label{lemma-adequate-charp}
 Let $A \hookarr B$ be an inclusion of $\FF_p$-algebras.  The following are equivalent:
\begin{enumerate1}
\item The homomorphism $A \hookarr B$ is universally adequate.
\item For every $b \in B$, there exists $r > 0$ such that $b^{\, p^r} \in A$.
\end{enumerate1}
Furthermore, if $A \hookarr B$ is finite type, then the above conditions are also equivalent to:
\begin{enumerate1} \setcounter{enumi}{2}
\item There exists $r > 0$ such that for all $b \in B$, $b^{\, p^r} \in A$.
\end{enumerate1}
In particular, an inclusion of fields is universally adequate if and only if it is purely inseparable.
\end{lem}

\begin{proof} Since condition (2) is easily seen to be stable under arbitrary base change,  we have $(2) \implies (1)$.  For $(1) \implies (2)$, we first show that a universally adequate inclusion of fields $k \hookarr k'$ is purely inseparable.  Indeed, let $k \hookarr k'$ be a separable field extension which is universally adequate and let $\bar k$ denote an algebraic closure of $k$.  Then $\bar k \hookarr \bar k \tensor_k k'$ is adequate which implies by Lemma \ref{lemma-transcendental} that $\bar k \tensor_k k' = \bar k$ and that $k = k'$.  

\medskip \noindent
Now suppose that $A$ and $B$ are Artin rings.  We can immediately reduce to the case where $A$ is local with maximal ideal $\fm_A$.  Then $B$ is a local ring with maximal ideal $\sqrt{\fm_A B}$ by Lemma \ref{lemma-adequate-basic}.  Since $A/ \fm_A \hookarr B / \sqrt{\fm_A B}$ is universally adequate, it is a purely inseparable field extension.  Therefore, for $b \in B$, there exist $a \in A$ and $n > 0$ such that $b - a^{\, p^n} \in \sqrt{\fm_A B}$ but then for some $m> 0$, $(b-a^{\, p^n})^{\,p^m} = b^{\, p^m} - a^{\, p^{n+m}} = 0$.  

\medskip \noindent
In the general case, an element $b \in B$ determines an $\FF_p$-algebra homomorphism $\FF_p[x] \to B$. If this map is not injective, the image $B_0 \subseteq B$ is an Artin ring and since $A_0 \hookarr B_0$ is an adequate inclusion, there exists a prime power of $b$ in $A$.  Otherwise, denote $A_0 = \FF_p[x] \cap A$.  Since $\text{Frac}(A_0) \hookarr \FF_p(x)$ is a purely inseparable field extension $\text{Frac}(A_0) = \FF_p(x^q)$ for a prime power $q$.  Denote by $A[x^q] \subseteq \FF_p[x]$ the subring generated by $A_0$ and $x^q$.  Then $f: \Spec(A_0[x^q]) \to \Spec(A_0)$ is an isomorphism over the generic point.  Let $I  = \text{Supp}(A_0[x^q] / A_0) \subseteq A_0$. Since $A_0/I \to A_0[x^q] / I A_0[x^q]$ is an adequate extension of Artin rings, there exists a prime power $q'$ such that $(A_0[x^q] / I A_0[x^q])^{q'} \subseteq A_0/I$.  It follows that the inclusion $A_0 \hookarr A_0[x^{qq'}]$ is an isomorphism so $x^{qq'} \in A$.

\medskip \noindent
It is clear that $(3) \implies (2)$.  Conversely, if $A \to B$ is of finite type, let $b_1, \cdots, b_n \in B$ be generators for $B$ as an $A$-algebra.  Choose $r > 0$ such that $b_i^{\, p^r} \in A$.  Then $b^{\, p^r} \in A$ for all $b \in B$.
\end{proof}

\begin{remark} If $A \hookarr B$ is not of finite type, a universal $r$ as in Lemma \ref{lemma-adequate-charp}(3) cannot be chosen.  For instance, consider
$\FF_p[x_1^p, x_2^{p^2}, x_2^{p^3}, \ldots] \hookarr \FF_p[x_1, x_2, x_3, \ldots]$
\end{remark}

%%%%%%%%%%%%%
\subsection{Adequate homeomorphisms}

\begin{defn} \label{definition-adequate-homeomorphism}
 A morphism $f:X \to Y$ of algebraic spaces is an \emph{adequate homeomorphism} if $f$ is an integral, universal homeomorphism which is a local isomorphism at all points with a residue field of characteristic $0$.  A ring homomorphism $A \to B$ is an \emph{adequate homeomorphism} if $\Spec(B) \to \Spec(A)$ is.  If $\cX$ is an algebraic stack, a morphism $\cA \to \cB$ of quasi-coherent $\oh_{\cX}$-algebras is an \emph{adequate homeomorphism} if $\sSpec_{\cX} (\cB) \to \sSpec_{\cX} (\cA)$ is.
\end{defn}

\begin{remark} \label{rmk-adequate-homeomorphism}
   A morphism $f: X \to Y$ of algebraic spaces is a local isomorphism at $x \in X$ if there exists an open neighborhood $U \subseteq X$ containing $x$ such that $f|_U$ is an isomorphism
onto its image.  
If $f: X \to Y$ is a locally of finite presentation morphism of schemes, then $f$ is a local isomorphism at $x$ if and only if $\oh_{Y, f(x)} \to \oh_{X,x}$ is an isomorphism (\cite[I.6.5.4, IV.1.7.2]{ega}).  The property of being an adequate homeomorphism is stable under base change and descends in the fpqc topology.  Therefore the property also extends to representable morphisms of algebraic stacks.  We note that by \cite[Corollary 4.20]{rydh_submersions}, any separated universal homeomorphism of algebraic spaces is necessarily integral.
\end{remark}

\medskip \noindent
We first consider the characteristic $p$ case and offer a slight generalization of \cite[Proposition 6.6]{kollar_quotients}.

\begin{prop} \label{proposition-adequacy-charp}
 Let $A \to B$ be an homomorphism of $\FF_p$-algebras.  Then the following are equivalent:
\begin{enumerate1}
\item The morphism $\Spec(B) \to \Spec(A)$ is an integral universal homeomorphism.
\item The morphism $\Spec(B) \to \Spec(A)$ is an adequate homeomorphism.
\item The ideal $\ker(A \to B)$ is locally nilpotent and $A \to B$ is universally adequate.
\item The ideal $\ker(A \to B)$ is locally nilpotent and for every $b \in B$, there exist $r > 0$ and $a \in A$ such that $a \mapsto b^{p^r}$.
\end{enumerate1}
If $A \to B$ is finite type, then the above are also equivalent to:
\begin{enumerate1} \setcounter{enumi}{4}
\item The ideal $\ker(A \to B)$ is locally nilpotent and there exists $r> 0$ such that for all $b \in B$, there exists $a \in A$ such that $a \mapsto b^{\, p^r}$.
\end{enumerate1}
\end{prop}

\begin{proof}
By definition, we have $(1) \iff (2)$.   Lemma \ref{lemma-adequate-homeomorphism} shows that $(3) \implies (1)$.  Lemma \ref{lemma-adequate-charp} shows that $(3) \iff (4)$ as well as $(4) \iff (5)$ if $A \to B$ is of finite type..  We need to show that $(1) \implies (4)$.  We may assume that $A \hookarr B$ is injective.  For $b \in B$, there exists a finite type $A$-subalgebra $A \subseteq B_0 \subseteq B$ containing $b$.  Then $\Spec(B_0) \to \Spec(A)$ is an integral universal homeomorphism.  We may assume that $A \hookarr B$ is finite.  Suppose first that $A$ is a local ring so $B$ is also a local ring (Lemma \ref{lemma-adequate-basic}). Let $\fm_A$ and $\fm_B$ denote the maximal ideals.   Let $b_1, \ldots, b_n$ be generators for $B$ as an $A$-module.  Since $A/\fm_A \to B / \fm_B$ is a purely inseparable field extension, there exists $r > 0$ such that for each $i$, the image of $b_i^{\, p^r}$ in $B/\fm_A B$ is contained in $A/\fm_A$.  Let $B_r$ be the $A$-subalgebra of $B$ generated by $b_i^{\, p^r}$ giving inclusions $A \subseteq B_r \subseteq B$.  For each $i$, the image of $b_i^{\, p^r}$ in $(B_r/A) \tensor_A A/\fm_A$ is 0.  Therefore $(B_r/A) \tensor_A A/\fm_A  = 0$ so by Nakayama's lemma $A = B_r$.  For the general case, let $b \in B$.  For each $\fp \in \Spec(A)$, there exists $r$ and $a/g \in A_{\fp}$ such that $\frac{a}{g} \mapsto b^{\, p^k}$ in $A_{\fp} \tensor_A B$.  Since $\Spec(A)$ is quasi-compact, there exists $r > 0$ and a finite collection of functions $g_1, \ldots, g_s \in A$ generating the unit ideal such that for each $i$, $g_i b^{\, p^r} \in A$.  We may write $1 = f_1 g_1 + \cdots + f_s g_s$ with $f_i \in A$.  Therefore $b^{\, p^r} = f_1 g_1 b^{\, p^r} + \cdots + f_s g_s b^{\, p^r} \in A$ which establishes (4).  
\end{proof}

\begin{remark}  Note that if condition (5) is satisfied with $r > 0$, then for any $A$-algebra $A'$ and $b' \in A' \tensor_A B$, there exists $a' \in A'$ such that $a' \mapsto b^{\, p^r}$.  If $A$ is Noetherian, then condition (5) above is equivalent to requiring the existence of a factorization
$$X = \Spec(B) \to \Spec(A) \to X^{(q)}$$
where $X \to X^{(q)}$ is the geometric Frobenius morphism for some $q = p^r$. 
\end{remark}

\medskip \noindent
We now adapt the proof of \cite[Lemma 8.7]{kollar_quotients}.   %We note that the statement there is not correct as inclusions of fields with property \cite[(8.7.2)]{kollar_quotients} are not necessarily purely inseparable (see Remark \ref{remark-counter}).

\begin{prop} \label{proposition-adequate-equiv}
Let $A \to B$ be a finitely presented\footnote{The published version of this paper had a version of this statement without the finitely presented hypothesis, which was incorrect.  Namely, when $A \to B$ is not finitely presented, the implication $(2) \implies (1)$ may fail, e.g., 
  $\ZZ \to \ZZ[ \{x_p\}_{p \text{ prime}}]/(px_p, x_p^p)$.}
 homomorphism of rings.  Then the following are equivalent:
\begin{enumerate1}
\item The morphism $\Spec(B) \to \Spec(A)$ is an adequate homeomorphism.
\item The ideal $\ker(A \to B)$ is locally nilpotent, $\ker (A \to B) \tensor \QQ = 0$ and $A \to B$ is universally adequate.
\end{enumerate1}
% If $A \to B$ is of finite type, then the above conditions are also equivalent to:
\begin{enumerate1} \setcounter{enumi}{2}
\item The ideal $\ker(A \to B)$ is locally nilpotent, $\ker (A \to B) \tensor \QQ = 0$ and there exists $N > 0$ such that for every $A$-algebra $A'$ and $b' \in A' \tensor_A B$, there exists $a' \in A'$ such that $a' \mapsto b'^N $.
\end{enumerate1}
\end{prop}

\begin{proof} Let $B' = \im(A \to B)$ and consider the factorization $A \mapsonto B' \hookarr B$.  The statement is clear for $A \mapsonto B'$.  We may therefore reduce to the case where $A \hookarr B$ is injective.

\medskip \noindent
For $(2) \implies (1)$, $\Spec(B) \to \Spec(A)$ is a universal homeomorphism by Lemma \ref{lemma-adequate-homeomorphism}.  Since $A \to B$ is universally adequate, $A \otimes \QQ \to B \otimes \QQ$ is an isomorphism by Lemma \ref{lemma-adequate-char0}.  Since $A \to B$ is
finitely presented, it follows that $A \to B$ is an isomorphism at all points of characteristic $0$ (see Remark \ref{rmk-adequate-homeomorphism}). 

%  For $(1) \implies (2)$, let $b \in B$.  By taking a finitely generated $A$-subalgebra $B_0 \subseteq B$ containing $b$, we can reduce to the case where $A \hookarr B$ is of finite type.  In this case, we will show that $(1) \implies (3)$.  
The implication $(3) \implies (2)$ is immediate. 
 To see $(1) \implies (3)$, define $Q = B/A$.  Since $\Spec(B) \to \Spec(A)$ is an isomorphism in characteristic $0$, we have $Q \tensor_{\ZZ} \QQ = 0$.  Since $Q$ is a finite $A$-module, there exists $m > 0$ such that $m Q = 0$.  

\medskip \noindent
We claim that there exists $N > 0$ such that for all $A/mA$-algebras $A'$ and $b' \in A' \tensor_{A/mA} B/mB$, there exists $a' \in A'$ with $a' \mapsto b'^N$.  
Write $m = p_1^{n_1} \cdots p_k^{n_k}$.  There are decompositions $A/mA = A_1 \oplus \cdots \oplus A_k$ and $B/mB = B_1 \oplus \cdots \oplus B_k$ with $\Spec(B_i) \to \Spec(A_i)$ a finite universal homeomorphism of $\ZZ/p_i^{n_i}$-schemes.   If for each $i$, there exists $N_i$ with the desired property for $A_i \to B_i$, then $N = \prod_i N_i$ satisfies the claim.  Assume $m = p^n$ and that $A \hookarr B$ is an inclusion of $\ZZ/m\ZZ$-algebras.  By Lemma \ref{lemma-adequate-charp} since $\Spec(A/pA) \to \Spec(B/pB)$ is a finite universal homeomorphism, there exists $r > 0$ such that for all $b \in B/pB$, there exists $a \in A/pA$ with $a \mapsto b^{\, p^r}$.  Therefore, for any $b \in B$, we may write $b^{\, p^r} = a + pb_1 \in A + pB$.  Then
$$b^{p^{r+n}} = (a + p b_1)^{p^n} = a^{p^n} + \sum_{i > 0} p^ i \binom{p^n}{i} a^{p^n-i} b_1^i = a^{p^n} \in A$$
since $p^n$ divides $p^i  \binom{p^n}{i}$ for $i >0$.  
 Furthermore, the same argument applied to $A' \to A' \tensor_{A} B$ for an $A$-algebra $A'$ shows that the property holds with the same choice of $r$.  This establishes the claim.

\medskip \noindent
For any $A$-algebra $A'$, if $Q' = \coker(A' \to A' \tensor_A B)$, then $Q' = Q \tensor_A A'$ and $mQ' = 0$.  Therefore, for any $b' \in A' \tensor_A B$, there exists $a' \in A'/mA'$ with $a' \mapsto b'^N$ in $A' \tensor_A B/m(A' \tensor_A B)$ which shows that the image of $b'^N \in Q'$ is contained in $mQ' = 0$ and so there exists $a' \in A'$ with $a' \mapsto b'^N$.
\end{proof}

\begin{remark} We note that since property $(1), (2)$ or $(3)$ implies that $A \to B$ is integral,  $A \to B$ is of finite type if and only if $A \to B$ is finite.  If in addition $A \to B$ is injective, then $A$ is Noetherian if and only if $B$ is Noetherian.
\end{remark}

\begin{example}
The condition that a morphism $f: \Spec(A) \to \Spec(B)$ be an adequate homeomorphism is not equivalent to the ring homomorphism $B \to A$ being universally adequate with locally nilpotent kernel.  For instance, consider $\Spec(\QQ) \to \Spec(\QQ[\epsilon] / (\epsilon^2))$.
\end{example}

%%%%%%%%%%%
%%%%%%%%%%%%%%%%%%%%%%%%%%%%%%%%%%%%%%%%%%%%%%%%%
\subsection{Universally adequate $\oh_{\cX}$-algebra homomorphisms}

\begin{defn}  \label{definition-adequate-stack-algebra} 
Let $\cX$ be an algebraic stack.  A morphism $\cA \to \cB$ of quasi-coherent $\oh_{\cX}$-algebras is \emph{universally adequate} if for every object $(U \to \cX) \in \Le(\cX)$ and section $s \in \cB(U \to \cX)$, there are an \'etale cover $\{U_i \stackrel{g_i}{\to} U \}$ and integers $N_i>0$ and $t_i \in \cA(U_i \to \cX)$  such that $t_i \mapsto (g_i^*s)^{N_i}$.
\end{defn}

\begin{remark}  It is clear that this is a Zariski-local condition on $\cX$ and that the composition of two universally adequate morphisms is again universally adequate.  For an object $(U \to \cX) \in \Le(\cX)$ with $U$ quasi-compact and section $s \in \cB(U \to \cX)$, a universal $N$ can be chosen. \end{remark}

\begin{lem} \label{lemma-equiv-defn}
A morphism $\cA \to \cB$ of quasi-coherent $\oh_{\cX}$-algebras is universally adequate if and only if for every smooth morphism $\Spec(A) \to \cX$ and $s \in \cB(\Spec(A) \to \cX)$, there exist an \'etale surjective morphism $g: \Spec(A') \to \Spec(A)$, an integer $N > 0$ and $t \in \cA(\Spec(A') \to \cX)$ such that $t \mapsto (g^*s)^{n}$.
\end{lem}

\begin{proof}  This is clear.
%The direction $\implies$ is clear.  For the other direction, suppose $(U \to \cX) \in \Le(\cX)$ and $s \in \cB(U \to \cX)$.  Cover $U$ with affine schemes $U_i$.  Then by hypothesis, there exists an etale cover $q_i: U'_i \to U_i$ such that some power of $q_i^* (s|_{U_i})$ lifts.  The \'etale morphisms $\{U'_i \to U_i \to U\}$ provide the necessary etale cover.
\end{proof}

\begin{lem} \label{lemma-adequate-affine} 
 If $X = \Spec(R)$ is an affine scheme, a morphism of quasi-coherent $\oh_{X}$-algebras $\cA \to \cB$ is universally adequate if and only if $\Gamma(X, \cA) \to \Gamma(X, \cB)$ is universally adequate.
\end{lem}

\begin{proof} Let $A = \Gamma(X, \cA)$ and $B = \Gamma(X, \cB)$.  The ``if'' direction is clear since for any smooth $R$-algebra $R'$, the ring homomorphism $A \tensor_R R' \to B \tensor_R R'$ is adequate.    Conversely, by Remark \ref{remark-smooth}, it suffices to show that for each $n$, $A[x_1, \ldots, x_n] \to B[x_1, \ldots, x_n]$ is adequate.  For each $b' \in B[x_1, \ldots, x_n]$, the hypotheses imply that there exist a faithfully flat $R[x_1, \ldots,x_n]$-algebra $R'$, an integer $N > 0$ and $a' \in A \tensor_{R} R'$ such that $a' \mapsto b^N \tensor 1$ in $B \tensor_R R'$.  But this then implies as in Lemma \ref{lemma-adequate-descent} that there exists $a \in A[x_1, \ldots, x_n]$ such that $a \mapsto b^N$.
\end{proof}

\begin{lem} \label{lemma-affine-presentation}  Let $\cX$ be a quasi-compact algebraic stack and let $f: \Spec(R) \to \cX$ be a smooth presentation.  A morphism $\cA \to \cB$ of quasi-coherent $\oh_{\cX}$-algebras is universally adequate if and only if $\Gamma(\Spec(R), f^* \cA) \to \Gamma(\Spec(R), f^* \cB)$ is universally adequate.
\end{lem}

\begin{proof}
The ``if'' direction is clear.  The ``only if'' direction follows from the same proof as that of Lemma \ref{lemma-adequate-affine}.
\end{proof} 

\begin{lem}  \label{lemma-adequate-algebra-bc} 
 Let $f: \cX \to \cY$ be a morphism of algebraic stacks.   Suppose that $\cA \to \cB$ is a morphism of quasi-coherent $\oh_{\cY}$-algebras.  Then
\begin{enumerate1}
\item If $\cA \to \cB$ is universally adequate, then $f^*\cA \to f^* \cB$ is universally adequate. 
\item If $f$ is fpqc and $f^* \cA \to f^* \cB$ is universally adequate, then $\cA \to \cB$ is universally adequate.
\end{enumerate1}
\end{lem}

\begin{proof}  
We may assume that $\cX$ and $\cY$ are quasi-compact.  Let $q: \Spec(R) \to \cX$ and $\Spec(S) \to \Spec(R) \times_{\cY} \cX$ be smooth presentations.  This gives a 2-commutative diagram
$$\xymatrix{
\Spec(S) \ar[r]^{f'} \ar[d]^p	& \Spec(R) \ar[d]^q \\
\cX \ar[r]^f			& \cY
}$$
with $f': \Spec(S) \to \Spec(R)$ faithfully flat.  Then Lemma \ref{lemma-affine-presentation} implies that $\cA \to \cB$ and $f^* \cA \to f^* \cB$ are universally adequate if and only if $\Gamma(\Spec(R), q^* \cA) \to \Gamma(\Spec(R), q^* \cB)$ and $\Gamma(\Spec(S), p^*f^* \cA) \to \Gamma(\Spec(S), p^*f^* \cB)$ are universally adequate, respectively.  Part (1) is now clear and part (2) follows from Lemma \ref{lemma-adequate-descent}.
\end{proof}

\begin{lem}  Let $\cX$ be an algebraic stack and let $\cA \to \cB$ be a 
  finitely presented\footnote{As with Proposition \ref{proposition-adequate-equiv}, the published version did not assume
  that $A \to B$ is finitely presented and was incorrect.}
  homomorphism of quasi-coherent $\oh_{\cX}$-algebras.  Then the following are equivalent:
\begin{enumerate1}
\item The ideal sheaf $\ker(\cA \to \cB)$ is locally nilpotent, $\ker(\cA \to \cB) \tensor \QQ = 0$ and $\cA \to \cB$ is universally adequate.
\item The morphism $\cA \to \cB$ is an adequate homeomorphism.
\item The morphism $\sSpec_{\cX}(\cB) \to \sSpec_{\cX} (\cA)$ is an adequate homeomorphism.
\end{enumerate1}
\end{lem}

\begin{proof} This follows from the definitions and fpqc descent using Lemma \ref{lemma-adequate-algebra-bc} and Proposition \ref{proposition-adequate-equiv}.
\end{proof}

%%%%%%%%%%%%%%%%%%%%%%%
\section{Adequately affine morphisms}

\medskip \noindent
In this section, we introduce a notion characterizing affineness for \emph{non-representable} morphisms of algebraic stacks which is weaker than cohomological affineness and will be an essential property of adequate moduli spaces.  This notion was motivated by and captures the properties of a morphism $[\Spec(A) / G] \to \Spec(A^G)$ where $G$ is a reductive group.

\subsection{The definition and equivalences}
\begin{defn}  \label{definition-adequately-affine}
A quasi-compact, quasi-separated morphism $f: \cX \to \cY$ of algebraic stacks is \emph{adequately affine} if for every surjection $\cA \to \cB$ of quasi-coherent $\oh_{\cX}$-algebras, the push-forward $f_* \cA \to f_* \cB$ is universally adequate.   A quasi-compact, quasi-separated algebraic stack $\cX$ is \emph{adequately affine} if $\cX \to \Spec(\ZZ)$ is adequately affine.
\end{defn}

\begin{remark}  By Lemma \ref{lemma-adequate-affine}, a quasi-compact, quasi-separated algebraic stack $\cX$ is adequately affine if and only if for every surjection $\cA \to \cB$ of quasi-coherent $\oh_{\cX}$-algebras, the ring homomorphism $\Gamma(\cX, \cA) \to \Gamma(\cX, \cB)$ is universally adequate.  Even though the notion of adequacy is not stable under base change, the above notion is equivalent to the seemingly weaker requirement that for every surjection $\cA \to \cB$ of quasi-coherent $\oh_{\cX}$-algebras, the ring homomorphism $\Gamma(\cX, \cA) \to \Gamma(\cX, \cB)$ be adequate; see Lemma \ref{lemma-affine-base-equiv}(3).
\end{remark}

\begin{remark} A quasi-compact, quasi-separated morphism $\cX \to \Spec(A)$ is adequately affine if and only if $\cX$ is adequately affine if and only if $\cX \to \Spec(\Gamma(\cX, \oh_{\cX}))$ is adequately affine.
\end{remark}

\begin{remark} Recall from \cite[Section 3]{alper-good} that a quasi-compact, quasi-separated morphism $f: \cX \to \cY$ of algebraic stacks is said to be \emph{cohomologically affine} if the push-forward functor $f_*$ is exact on quasi-coherent $\oh_{\cX}$-modules.
\end{remark}

\begin{lem} \label{lemma-cohomologically-affine}  Let $f: \cX \to \cY$ be a quasi-compact, quasi-separated morphism of algebraic stacks.  
Then $f$ is cohomologically affine if and only if for every surjection 
$\cA \to \cB$ of quasi-coherent $\oh_{\cX}$-algebras, $f_* \cA \to f_* \cB$ is surjective.
\end{lem}

\begin{proof}
The ``only if" direction is clear.  Conversely, let $\cF \to \cG$ be a surjection of quasi-coherent $\oh_{\cX}$-modules.  Then $\Sym^* \cF \to \Sym^* \cG$ is a surjection of graded quasi-coherent $\oh_{\cX}$-algebras.   As $f_*$ is exact on the category of quasi-coherent $\oh_{\cX}$-modules, $f_* \Sym^* \cF \to f_*\Sym^* \cG$ is surjective and it follows that $f_* \cF \to f_* \cG$ is surjective.
\end{proof}

\medskip \noindent
We now establish a key lemma which will be used to prove that adequate moduli spaces and good moduli spaces are equivalent notions in characteristic $0$ (see Proposition \ref{proposition-good}).

\begin{lem} \label{lemma-adequately-affine-char0}
Let $f: \cX \to \cY$ be a quasi-compact, quasi-separated morphism of algebraic stacks defined over $\Spec(\QQ)$.  Then $f$ is adequately affine if and only if $f$ is cohomologically affine.
\end{lem}

\begin{proof} This follows from Lemmas \ref{lemma-cohomologically-affine} and \ref{lemma-adequate-char0}.
\end{proof}

\begin{lem} \label{lemma-adequately-affine-equiv}
Let $f: \cX \to \cY$ be a quasi-compact, quasi-separated morphism of algebraic stacks.  The following are equivalent:
\begin{enumerate1}
\item For every universally adequate morphism $\cA \to \cB$ of quasi-coherent $\oh_{\cX}$-algebras with kernel $\cI$, the morphism $f_* \cA /f_* \cI \to f_* \cB$ is an adequate homeomorphism.
\item The morphism $f$ is adequately affine.
\item For every surjection $\cF \to \cG$ of quasi-coherent $\oh_{\cX}$-modules, the morphism $f_* \Sym^* \cF \to f_* \Sym^* \cG$ is universally adequate.
\end{enumerate1}
\medskip \noindent
If in addition $\cX$ is Noetherian, then the above are equivalent to:
\begin{enumerate1} 
\item[$(1')$] For every universally adequate morphism $\cA \to \cB$ of finite type quasi-coherent $\oh_{\cX}$-algebras with kernel $\cI$, the mophism $f_* \cA/f_* \cI \to f_* \cB$ is an adequate homeomorphism.
\item[$(2')$] For every surjection $\cA \to \cB$ of finite type quasi-coherent $\oh_{\cX}$-algebras, the morphism $f_* \cA \to f_* \cB$ is universally adequate.  
\item[$(3')$] For every surjection $\cF \to \cG$ of coherent $\oh_{\cX}$-modules,  the morphism $f_* \Sym^* \cF \to f_* \Sym^* \cG$ is universally adequate.
\end{enumerate1}
\end{lem}

\begin{proof}
It is obvious that $(1) \implies (2) \implies (3)$.  We now show that $(3) \implies (2) \implies (1)$.  Suppose that property (3) holds and let $\cA \to \cB$ be a surjection of quasi-coherent $\oh_{\cX}$-algebras.  The natural map $\Sym^* \cB \to \cB$ has a section so that $f_* \Sym^* \cB \to f_* \cB$ is surjective.  There is  a commutative diagram
$$\xymatrix{
f_* \Sym^* \cA \ar[r]\ar[d]		& f_* \Sym^* \cB \ar[d] \\
f_* \cA \ar[r]					& f_* \cB
}$$
Since the composition $f_* \Sym^* \cA \to f_* \Sym^* \cB \to f_* \cB$ is universally adequate, so is $f_* \cA \to f_* \cB$, which establishes property (2).  Suppose that property (2) holds.  We may assume that $\cX$ and $\cY$ are quasi-compact.  Let $\cA \to \cB$ be a universally adequate morphism of quasi-coherent $\oh_{\cX}$-algebras.  Let $\cB' = \im(\cA \to \cB)$.  Then $f_* \cA \to f_* \cB$ is universally adequate so we may assume that $\cA \to \cB$ is injective.  Let $V \to \cY$ and $U \to  \cX_V:=\cX \times_{\cY} V$ be smooth presentations with $U$ and $V$ affine.  Let $R = U \times_{\cX_V} U$.  This gives a diagram
$$\xymatrix{
U \times_{\cX_V} U \ar@<0.5ex>[r] \ar@<-.5ex>[r]	&U \ar[r] & \cX_V \ar[r] \ar[d]  \ar @{} [dr] |{\square}		& V \ar[d] \\
	&& \cX \ar[r]		& \cY
}$$
We have a diagram of exact sequences
$$\xymatrix{
f_* \cA (V \to \cY) \ar@{^(->}[r] \ar@{^(->}[d]	& \cA(U \to \cX) \ar@<0.5ex>[r] \ar@<-0.5ex>[r] \ar@{^(->}[d]  & \cA(U \times_{\cX_V} U \to \cX)\ar@{^(->}[d] \\
f_* \cB (V \to \cY) \ar@{^(->}[r] 				& \cB(U \to \cX) \ar@<0.5ex>[r] \ar@<-0.5ex>[r]  & \cB(U\times_{\cX_V} U \to \cX)
}$$
Since $\cA \to \cB$ is universally adequate, Lemma \ref{lemma-affine-presentation} implies that the middle vertical arrow is universally adequate.  Therefore, for $s \in f_*\cB(V \to \cY)$, there exist $N > 0$ and $t \in \cA(U \to \cX)$ with $t \mapsto s^N$.  By exactness, we must have $t \in f_* \cA(V \to \cY)$.  Therefore $f_* \cA(V \to \cY) \to f_* \cB(V \to \cY)$ is universally adequate which establishes that $f_* \cA \to f_* \cB$ is a universally adequate.  Statement (1) follows.

\medskip \noindent
In the locally Noetherian case, direct limit methods imply that for each $i \in \{1,2,3\}$,    $(i) \iff (i')$.  We spell out the details only for $(2) \iff (2')$.   Given an arbitrary surjective morphism of $\oh_{\cX}$-algebras $\alpha: \cF \to \cG$, we apply Lemma \ref{lemma-qcoh-limit} to write $\cG = \dlim \cG_{\alpha}$ with each $\cG_{\alpha} \subseteq \cG$ a finite type $\oh_{\cX}$-algebra.  The inverse $\cF_{\alpha} = \alpha^{-1}(\cG_{\alpha})$ is a quasi-coherent $\oh_{\cX}$-algebra.  If we knew the proposition for $\cG$ finite type, then each $f_* \cF_{\alpha} \to f_* \cG_{\alpha}$ is universally adequate.  Given $(\Spec(B) \to \cY) \in \Le(\cY)$ and $s \in f_* \cG(\Spec(B) \to \cY)$, then as $f_*\cG(\Spec(B) \to \cY) = \dlim f_* \cG_{\alpha} (\Spec(B) \to \cY)$, there exists $\alpha$ such that $s \in f_* \cG_{\alpha}(\Spec(B) \to \cY)$.  But then there exist $N>0$ and $t \in f_* \cF_{\alpha}(\Spec(B) \to \cY)$ with $t \mapsto s_{\alpha}^N$.  We may now assume $\cG$ is a finite type $\oh_{\cX}$-algebra.  By apply Lemma \ref{lemma-qcoh-limit} again, we may write $\cF = \dlim \cF_{\alpha}$.  Then there exists $\alpha$ such that $\cF_{\alpha} \to \cG$ is surjective and $f_* \cF_{\alpha} \to f_* \cG$ is universally adequate which implies that $f_* \cF \to f_* \cG$ is universally adequate.  
\end{proof}

\begin{lem} \label{lemma-affine-base-equiv}
Let $\cX$ be a quasi-compact and quasi-separated algebraic stack.  The following are equivalent:
\begin{enumerate1}
\item For every universally adequate morphism $\cA \to \cB$ of quasi-coherent $\oh_{\cX}$-algebras with kernel $\cK$, the induced algebra homomorphism $\Gamma(\cX, \cA) / \Gamma(\cX, \cK) \to \Gamma(\cX, \cB)$ is an adequate homeomorphism.
\item The algebraic stack $\cX$ is adequately affine.
\item For every surjection $\cA \to \cB$ of quasi-coherent $\oh_{\cX}$-algebras, the homomorphism $\Gamma(\cX, \cA) \to \Gamma(\cX, \cB)$ is adequate. 
\item For every surjection $\cF \to \cG$ of quasi-coherent $\oh_{\cX}$-modules, the homomorphism $\Gamma(\cX, \Sym^* \cF) \to \Gamma(\cX, \Sym^* \cG)$ is adequate.
\item For every surjection $\cF \to \oh_{\cX}$ of quasi-coherent $\oh_{\cX}$-modules, there exist $N> 0$ and $f \in \Gamma(\cX, \Sym^N \cF)$ such that $f \mapsto 1$ under $\Gamma(\cX, \Sym^N \cF) \to \Gamma(\cX, \oh_{\cX})$.
\end{enumerate1}
If in addition $\cX$ is Noetherian, then the above are equivalent to:
\begin{enumerate1}  
\item[$(1')$] For every universally adequate morphism $\cA \to \cB$ of finite type quasi-coherent $\oh_{\cX}$-algebras with kernel $\cK$, the induced algebra homomorphism $\Gamma(\cX, \cA)/ \Gamma(\cX, \cK) \to \Gamma(\cX, \cB)$ is universally adequate.
\item[$(2')$] For every surjection $\cA \to \cB$ of finite type quasi-coherent $\oh_{\cX}$-algebras, then $ \Gamma(\cX, \cA) \to \Gamma(\cX, \cB)$ is universally adequate.
\item[$(3')$] For every surjection $\cA \to \cB$ of finite type quasi-coherent $\oh_{\cX}$-algebras, then $ \Gamma(\cX, \cA) \to \Gamma(\cX, \cB)$ is adequate.
\item[$(4')$] For every surjection $\cF \to \cG$ of coherent $\oh_{\cX}$-modules, $\Gamma(\cX, \Sym^* \cF) \to \Gamma(\cX, \Sym^* \cG)$ is adequate.
\item[$(5')$] For every surjection $\cF \to \oh_{\cX}$ of coherent $\oh_{\cX}$-modules, there exist $N> 0$ and $f \in \Gamma(\cX, \Sym^N \cF)$ such that $f \mapsto 1$ under $\Gamma(\cX, \Sym^N \cF) \to \Gamma(\cX, \oh_{\cX})$.
\end{enumerate1}
If in addition $\cX$ has the resolution property (that is, for every coherent $\oh_{\cX}$-module $\cF$, there exists a surjection $\cV \to \cF$ from a locally free $\oh_{\cX}$-module of finite rank), then the above are equivalent to
\begin{enumerate1}
\item[$(5')$] For every surjection $\cV \to \oh_{\cX}$ from a locally free $\oh_{\cX}$-module of finite rank, there exist $N> 0$ and $f \in \Gamma(\cX, \Sym^N \cV)$ such that $f \mapsto 1$ under $\Gamma(\cX, \Sym^N \cV) \to \Gamma(\cX, \oh_{\cX})$.
\end{enumerate1}
\end{lem}

\begin{proof}  
It is immediate that $(1) \implies (2) \implies (3) \implies (4) \implies (5)$.  Lemma \ref{lemma-adequately-affine-equiv} shows that $(2) \implies (1)$.  For $(3) \implies (2)$, let $\cA \to \cB$ be a surjection of quasi-coherent $\oh_{\cX}$-algebras.  By Remark \ref{remark-smooth}, it suffices to show that for each $n$,  the homomorphism $\Gamma(\cX, \cA)[x_1, \ldots, x_n] \to \Gamma(\cX, \cB)[x_1, \ldots, x_n]$ is adequate, but this corresponds to 
$$\Gamma(\cX, \cA \tensor_{\oh_{\cX}} \oh_{\cX}[x_1, \ldots, x_n]) \to \Gamma(\cX, \cB \tensor_{\oh_{\cX}} \oh_{\cX}[x_1, \ldots, x_n])$$
 which is adequate by statement (3).  The same argument of Lemma \ref{lemma-adequately-affine-equiv} shows that $(4) \implies (3)$.  For $(5) \implies (3)$, suppose that $\cA \to \cB$ is a surjection of quasi-coherent $\oh_{\cX}$-algebras.   A section $s \in \Gamma(\cX, \cB)$ gives a morphism of $\oh_{\cX}$-modules $\oh_{\cX} \to \cB$. Consider the fiber product and the induced diagram
$$
\begin{array}{ccc}
\xymatrix{
\cA \times_{\cB} \oh_{\cX} \ar[r] \ar[d]		& \oh_{\cX} \ar[d]^s \\
\cA \ar[r]							& \cB	 
} \qquad
\xymatrix{
\Sym^* (\cA \times_{\cB} \oh_{\cX}) \ar[r] \ar[d]		& \Sym^* \oh_{\cX} \cong \oh_{\cX}[x] \ar[d]	& x \ar@{|->}[d] \\
\cA \ar[r]							& \cB & s
}
\end{array}$$
There exist $N > 0$ and $\tilde t \in \Gamma(\cX,\Sym^* (\cF \times_{\cG} \oh_{\cX}))$ with $\tilde t \mapsto x^N$ under $\Gamma(\cX,\Sym^* (\cF \times_{\cG} \oh_{\cX})) \to \Gamma(\cX, \Sym^* \oh_{\cX}) \cong \Gamma(\cX, \oh_{\cX})[x]$.   If $t$ is the image of $\tilde t$ under the composition 
$$\Gamma(\cX,\Sym^* (\cA \times_{\cB} \oh_{\cX})) \to \Gamma(\cX, \Sym^* \cA),$$
 then $t \mapsto s^N$ which establishes statement (3).  Direct limit methods show the equivalences of $(i) \iff (i')$ for $i \in \{1, \ldots, 5\}$. The equivalence of $(5') \iff (5'')$ is immediate.
\end{proof}

%%%%%%%%%%
\subsection{Properties of adequately affine morphisms}

\begin{prop} \label{proposition-adequately-affine} \quad
\begin{enumerate1}
\item 
Adequately affine morphisms are stable under composition.
\item 
A cohomologically affine morphism $f: \cX \to \cY$ of algebraic stacks is adequately affine.  In particular, an affine morphism $f: \cX \to \cY$ of algebraic stacks is adequately affine.
\item
If $f: \cX \to \cY$ is an adequately affine morphism of algebraic stacks over an algebraic space $S$ and $S' \to S$ is a morphism of algebraic spaces, then $f_{S'} = \cX_{S'} \to \cY_{S'}$ is adequately affine. 
\end{enumerate1}
Consider a 2-cartesian diagram of algebraic stacks:
$$\xymatrix{ 
\cX' \ar[r]^{f'} \ar[d]^{g'} \ar @{} [dr] |{\square}	& \cY' \ar[d]^g\\
\cX \ar[r]^f				& \cY
}$$
\begin{enumerate1} \setcounter{enumi}{3}
\item 
If $g$ is faithfully flat and $f'$ is adequately affine, then $f$ is adequately affine.
\item
If $f$ is adequately affine and $g$ is a quasi-affine morphism, then $f'$ is adequately affine.
\item
 If $f$ is adequately affine and $\cY$ has quasi-affine diagonal over $S$, then $f'$ is adequately affine. In particular, if $\cY$ is a Deligne$-$Mumford stack with quasi-compact and separated diagonal, then $f$ adequately affine implies that  $f'$ is adequately affine.\\
\end{enumerate1}
\end{prop}

\begin{proof}  Part (1) follows from Proposition \ref{lemma-adequately-affine-equiv}.  Part (2) is clear.  For part (4), suppose that $\alpha: \cA \to \cB$ is a surjection of $\oh_{\cX}$-algebras.  Since $g'^*$ is exact and $f'$ is adequately affine, $f'_* g'^* \alpha$ is universally adequate.  By flat base change, $g^* f_* \alpha$ is canonically identified with $f'_* g'^* \alpha$.  By Lemma \ref{lemma-adequate-algebra-bc}(2), $f_* \alpha$ is universally adequate.  Therefore $f$ is adequately affine.  

\medskip \noindent
For part (5), let $\alpha: \cA' \to \cB'$ be a surjection of $\oh_{\cX'}$-algebras.  Suppose first that $g: \cY' \to \cY$ is a quasi-compact open immersion.   Let $\tilde \cB = \im(g'_* \cA \to g'_* \cB)$.  Since $g'^* g'_* \cA' \cong \cA'$ and $g'^* g'_* \cB' \cong \cB'$, there is a factorization $\cA' \mapsonto g'^* \tilde \cB \hookarr \cB'$ and we conclude that there is a canonical isomorphism $g'^* \tilde \cB \cong \cB'$.  Since $f$ is adequately affine, $f_* g'_* \cA' \to f_* \tilde \cB$ is universally adequate.  By Lemma \ref{lemma-adequate-algebra-bc}(1), $g^* f_* g'_* \cA' \to g^* f_* g'_* \tilde \cB$ is universally adequate but this is identified with $f'_* g'^* g'_* \cA' \to f'_* g'^* g'_* \tilde \cB$ which is identified with $f'_* \cA' \to f'_* \cB'$.   Now suppose that $g: \cY' \to \cY$ is an affine morphism so that the functors $g_*$ and $g'_*$ are faithfully exact on quasi-coherent sheaves.  It is also easy to see that a morphism $\cC \to \cD$ of quasi-coherent $\oh_{\cY'}$-algebras is universally adequate if and only if $g_*\cC \to g_* \cD$ is.  Since $f$ is adequately affine, $f_* g'_* \alpha \cong g_* f'_* \alpha$ is universally adequate and it follows that $f'_* \alpha$ is universally adequate.  This establishes part (5).

\medskip \noindent
For part (6), the question is Zariski-local on $\cY$ and $\cY'$ so we may assume that they are quasi-compact.  Let $Y  \to \cY$ be a smooth presentation with $Y$ affine.  Since $\Delta_{\cY/S}$ is quasi-affine, $Y\to \cY$ is a quasi-affine morphism.  We may choose a smooth presentation $Z \to \cY'_Y:= \cY' \times_{\cY} Y$ with $Z$ an affine scheme.  We have the 2-cartesian diagram:
$$\xymatrix{
				& \cZ \ar[d] \ar[rr]			&				& Z \ar[d] \\
				&\cX'_Y \ar[rr] \ar[dd] \ar[dl]			&				& \cY'_{Y} \ar[dd] \ar[dl] \\
\cX' \ar[rr]\ar[dd] &						& \cY' \ar[dd]	& \\
				& \cX_Y \ar[rr] \ar[dl]				& 				& Y \ar[dl] \\
\cX  \ar[rr]	&							& \cY 		&
}$$
Since $\cX \to \cY$ is adequately affine and $Y \to \cY$ is a quasi-affine morphism, by part (5) $\cX_Y \to Y$ is adequately affine.   The morphism $Z \to Y$ is affine which implies that $\cZ \to Z$ is adequately affine.  Since the composition $Z \to \cY'_Y \to \cY'$ is smooth and surjective, by descent $\cX' \to \cY'$ is adequately affine.  For the final statement of part (6), $\Delta_{\cY/S}: \cY \to \cY \times_S \cY$ is separated, quasi-finite and finite type so by Zariski's Main Theorem for algebraic spaces, 
$\Delta_{\cY/S}$ is quasi-affine.
Finally, part (3) follows formally from parts (4) and (5).
\end{proof}

\begin{remark}  \label{remark-quasi-affine}
Part (5) can fail if $\cY' \to \cY$ is not quasi-affine and part (6) can fail if $\cY$ does not have quasi-affine diagonal.  As in the example given in \cite[Remark 3.11]{alper-good}, if $A$ is an abelian variety over a field $k$, then $\Spec(k) \to BA$ is cohomologically affine (and therefore adequately affine) but $A = \Spec(k) \times_{BA} \Spec(k) \to \Spec(k)$ is not adequately affine.
\end{remark}

\begin{lem} \label{lemma-propertyP}
Let $f: \cX \to \cY$, $g: \cY \to \cZ$ be morphisms of algebraic stacks where either $g$ is quasi-affine or $\cZ$ has quasi-affine diagonal over $S$.  Suppose that $g \circ f$ is adequately affine and that $g$ has affine diagonal.  Then $f$ is adequately affine.
\end{lem}

\begin{proof}
 This is clear from the 2-cartesian diagram
$$\xymatrix{
				& \cX \ar[r]^{(id,f)} \ar[dl]	& \cX \times_{\cZ} \cY \ar[r]^{p_2}	 \ar[dl] \ar[dr] & \cY \ar[dr]^g \\
\cY \ar[r]^{\Delta}	& \cY \times_{\cZ} \cY	&	&  \cX \ar[r]				& \cZ 
}$$
and Proposition \ref{proposition-adequately-affine}.
\end{proof}

\subsection{Generalization of Serre's criterion}

\begin{thm} \label{serres-criterion}
A quasi-compact, quasi-separated morphism $f: X \to Y$ of algebraic spaces is adequately affine if and only if it is affine.
\end{thm}

\begin{proof} By Proposition \ref{proposition-adequately-affine}, we may assume that $Y$ is an affine scheme.  We first show that the proof of \cite[II.5.2.1]{ega} generalizes when $X$ is a scheme.  Set $R = \Gamma(X, \oh_X)$.  For a closed point $q \in X$, let $U$ be an open affine neighborhood of $q$ with $Y = X \setminus U$.    Consider the surjective morphism of quasi-coherent $\oh_X$-algebras
$$\Sym^* \cI_Y \to \Sym^* k(q) \cong k(q)[x]$$
Since $X$ is adequately affine, there exist an integer $N$ and $f' \in \Gamma(X, \Sym^* \cI_Y)$ with $f' \mapsto x^N$.  Let $f\in R$ be the image of $f'$ under $\Gamma(X, \Sym^* \cI_Y) \to \Gamma(X,\oh_X)=R$.  We have $q \in X_f \subseteq U$.  Furthermore, $X_f$ is an affine scheme since $X_f = U_{f}$.

\medskip \noindent
Since $X$ is quasi-compact, we may find functions $f_1, \ldots, f_n \in R$ such that the affine schemes $X_{f_i}$ cover $X$.  Since affineness is Zariski-local, it suffices to show that $f_1, \ldots, f_n$ generate the unit ideal of $R$.  There is a surjection of $\oh_X$-algebras
$$\alpha:  \oh_X [t_1, \ldots, t_k] \to \oh_X[x]$$
defined by sending $t_i$ to $f_i x$.  Therefore
$$\Gamma(\alpha): R[t_1, \ldots, t_k] \to R[x]$$
is adequate and there exist an integer $N > 0$ and $g \in R[t_1, \ldots, t_k]$ of degree $N$ such that $g \mapsto x^{N}$.  But this implies that the monomials of $\prod_i f_i^{n_i}$ of degree $N$ generate the unit ideal and thus $(f_i) = R$.  

\medskip \noindent
In general, if $X$ is an algebraic space, by \cite[Theorem B]{rydh_noetherian}, there exists a finite surjective morphism $X' \to X$ from a scheme $X'$.  Since $X'$ is adequately affine, $X'$ is affine.
By Chevalley's criterion for algebraic spaces (see \cite[Corollary A.2]{conrad-deligne} or \cite[Theorem 8.1]{rydh_noetherian}), $X$ is affine.
\end{proof}

\begin{cor} A quasi-compact, quasi-separated representable morphism $f: \cX \to \cY$ of algebraic stacks where $\cY$ has quasi-affine diagonal is adequately affine if and only if it is affine.
\end{cor}

\begin{proof}
This follows from Proposition \ref{proposition-adequately-affine} and Theorem \ref{serres-criterion}.
 \end{proof}

\begin{remark}  As in Remark \ref{remark-quasi-affine}, the corollary can fail if $\cY$ does not have quasi-affine diagonal; if $A$ is an abelian variety over a field $k$, then $\Spec(k) \to BE$ is adequately affine but not affine.
\end{remark}

%%%%%%%%%%%%%
\section{Adequate moduli spaces} \label{section-adequate-moduli}

We introduce the notion of an adequate moduli space and then prove its basic properties. 

\subsection{The definition}

\begin{defn} \label{definition-adequate-moduli-space}  A quasi-compact, quasi-separated morphism $\phi: \cX \to Y$ from an algebraic stack $\cX$ to an algebraic space $Y$ is called an \emph{adequate moduli space} if the following properties are satisfied:
\begin{enumerate1}
\item The morphism $\phi$ is adequately affine,.
\item The natural map $\oh_Y \iso \phi_* \oh_{\cX}$ is an isomorphism. 
\end{enumerate1}
\end{defn}

\begin{remark}  A quasi-compact, quasi-separated morphism $p: \cX \to S$ from an algebraic stack to an algebraic space $S$ is adequately affine if and only if the natural map $\cX \to \sSpec( p_* \oh_{\cX})$ is an adequate moduli space. 
\end{remark}

\begin{remark}  As in \cite[Remark 4.4]{alper-good}, one could also consider the relative notion for an arbitrary quasi-compact, quasi-separated morphisms of algebraic stacks $\phi: \cX \to \cY$ satisfying the two conditions in Definition \ref{definition-adequate-moduli-space}.
\end{remark} 

\medskip \noindent 
In characteristic $0$, the notions of good moduli spaces and adequate moduli spaces agree.

\begin{prop} \label{proposition-good}
A quasi-compact, quasi-separated morphism $\phi: \cX \to Y$ over $\Spec(\QQ)$ from an algebraic stack $\cX$ to an algebraic space $Y$ is a good moduli space if and only if it is an adequate moduli space.
\end{prop}

\begin{proof}  
This follows from Lemma \ref{lemma-adequately-affine-char0}.
\end{proof}

\subsection{First properties}
We establish the basic properties of adequate moduli spaces as well as provide examples where the correspondingly stronger property of good moduli spaces does not hold.

\begin{lem} \label{lemma-ideal}
Suppose that $\phi: \cX \to Y$ is an adequately affine morphism of algebraic stacks where $Y$ is an algebraic space. Let $\cA$ be a quasi-coherent $\oh_{\cX}$-algebra and let $\cI$ be a quasi-coherent sheaf of $\oh_{Y}$-ideals.  Then
\begin{equation} \label{equation-ideal}
\phi_* \cA / \cI \to \phi_*(\cA/ \cI \cA)
\end{equation}
is an adequate homeomorphism.
\end{lem}

\begin{proof}
The quasi-coherent sheaf $\cI \cA$ is the image of $\phi^* \cI \to \cA$.  The surjection $\cA \to \cA / \cI \cA$ induces an adequate homeomorphism $\phi_* \cA / \phi_* \cI \cA \to \phi_* (\cA / \cI \cA)$ since $\phi$ is adequately affine. 
It suffices to show that the surjection $\phi_* \cA / \cI \to \phi_* \cA / \phi_* (\cI \cA)$ is an adequate homeomorphism.  Since it is an isomorphism in characteristic $0$,  it suffices to show that the kernel is locally nilpotent.  Since the question is local in the fpqc topology, we may assume that $Y $ is an affine scheme; let $A = \Gamma(Y, \cA)$ and $I = \Gamma(Y, \cI)$.  A choice of generators $f_j$ for $j \in J$ of $I \subseteq A$ induces a surjection $A[x_j; j \in J] \to \bigoplus_{n \ge 0} I^n $ where $x_j \mapsto f_j$.  This induces a surjection $\cA[x_j; j \in J] \to \bigoplus_{n \ge 0} \cI^n \cA$.  Since $\phi$ is adequately affine, 
$$
A[x_j; j \in J] \to \bigoplus_{n \ge 0} \Gamma(\cX, \cI^n \cA)
$$
is adequate which shows that for every $f \in \Gamma(\cX, \cI \cA)$, there exists $N > 0$ such that $f^N \in I$.  
\end{proof}

\begin{remark} \label{remark-git}
 Let $S$ be an affine scheme and let $G$ be a geometrically reductive group scheme over $S$ (see Section \ref{section-adequate-groups}) acting on an affine scheme $\Spec(R)$.  Let $I \subseteq R^G$ be an ideal.  Then Lemma \ref{lemma-ideal} implies that the map
 $$ R^G / I \to (R/IR)^G$$
 is an adequate homeomorphism. 
 \end{remark}

\begin{example} \label{example1}
This example shows that the map (\ref{equation-ideal}) need not be surjective.
Consider the action of $\ZZ / p \ZZ$ on $\AA^2 = \FF_p[x,y]$ over $\FF_p$ where a generator acts by $(x,y) \mapsto (x+y, y)$.   Let $z = x(x+y) \cdots (x+ (p-1)y)$.  Then 
$$\phi: \cX = [\AA^2 / \ZZ/p\ZZ] \to \Spec(\FF_p[y,z]) = Y$$
is an adequate moduli space (see Theorem \ref{theorem-git})   
and the map (\ref{equation-ideal}) applied with the ideal $(y)$ corresponds to
$$\FF_p[x^p] \cong \FF_p[y,z]/(y) \to (\FF_p[x,y]/ (y) \FF_p[x,y])^G \cong \FF_p[x]$$
which is not surjective.
\end{example}

\begin{example} \label{example2}
This example shows that the map (\ref{equation-ideal}) need not be injective.
Consider the action of $\ZZ / p \ZZ$ on $X = \Spec(R) = \FF_p[x_1, x_2,y]/(x_1 x_2)$ over $\FF_p$ where a generator acts by $(x_1, x_2, y) \mapsto (x_1, x_2, x_1 + y)$.  Let $I = (x_1, x_2)$.  Then 
 the invariant $x_2y \in IR \cap R^G$ is not in $I$.  That is, $x_2 y$ is a non-zero element in the kernel of $R^G/ I \to (R/IR)^G$.

\end{example}

\begin{example} \label{example-johan}
The following example due to Johan de Jong shows that in the definition of an adequately affine morphism $\cX \to \cY$,
the degree of the exponent required to lift sections cannot be universally bounded over all quasi-coherent $\oh_{\cX}$-algebras.  In particular, for a geometrically reductive group scheme (see Section \ref{section-adequate-groups}), the degree of the exponent required to lift invariant sections cannot be universally bounded over all $G$-modules.  However, for finite flat group schemes $G \to S$, such a universal bound can be chosen.

Consider the geometrically reductive group $\SL_2$ over $\FF_2$.  We show that there does not exist $N > 0$ such that for every surjection $V \to \FF_2$ of $\SL_2$-representations, the map $(\Sym^{N}V)^{\SL_2} \to \FF_2$ is non-zero.  
Let $W = \FF_2 x \oplus \FF_2 y$ be the standard representation of $\SL_2$.  For each $n > 0$, consider 
the representation
$$ V_n = \Sym^{2(2^n - 1)}(W)$$
which has basis elements $Z_{i,j} = x^iy^j$ where $i + j = 2(2^n - 1)$ where $i,j \ge 0$.  Consider the $\SL_2$-equivariant surjection 
$$
V_n \to k, \quad
Z_{i,j} \mapsto \left \{ 
\begin{array}{rl} 
1 & \text{if } i=j=2^n -1 \\
0& \text{otherwise.}
\end{array} \right.
$$
If $\gamma
= \begin{pmatrix} 1 & t \\ 0 & 1 \end{pmatrix} \in \SL_2
$, one can check that
$$\gamma \cdot Z_{2^n - 1, 2^n - 1} =
Z_{2^n - 1, 2^n - 1} + t Z_{2^n - 2, 2^n} + t^2 Z_{2^n - 3, 2^n + 1} +
\cdots + t^{2^n - 1}Z_{0, 2(2^n - 1)}.$$
\medskip \noindent
Suppose that for some $d > 0$, there exists a non-zero invariant element 
$$
v_n
=
(Z_{2^n - 1, 2^n - 1})^d + \sum_{a < d} (Z_{2^n - 1, 2^n - 1})^a
F_a(Z_{i,j}; i \neq j) \in (\Sym^d V_n)^{\SL_2}$$
for some elements $F_a \in \Sym^{d-a} V_n$.  We claim that $d \ge 2^n$.
The coefficient of $(Z_{2^n-2,2^n})^d$ in $\gamma \cdot (Z_{2^{n}-1,2^n-1})^d$ is $t^d$.  By expanding $\gamma \cdot v_n$ and by considering the coefficient of $t^d (Z_{2^n-2,2^n})^d$ in the equality $v_n = \gamma \cdot v_n$, there must exist some $a$ such that the coefficient of $t^d (Z_{2^n-2,2^n})^d$ in  $(Z_{2^n - 1, 2^n - 1})^a
F_a(Z_{i,j}; i \neq j)$ is non-zero.
 The coefficient of $Z_{2^n-2,2^n}$ in $\gamma \cdot Z_{i,j}$ is a non-zero only if $i \ge 2^n-2$ in which case the corresponding term is
 $$\binom{i}{2^n -2} t^{i-2^n+2} Z_{2^n-2,2^n}.$$
One can check that this binomial coefficient is divisible by $2$ for all $i$ in the range $2^n - 1 < i < 2(2^n-1)$.  When $i=2(2^n-1)$, this term is a multiple of $t^{2^n}Z_{2^n - 2, 2^n}$.  If $d < 2^n$, there is thus never a non-zero coefficient 
of  $t^d (Z_{2^n-2,2^n})^d$ in 
$(Z_{2^n - 1, 2^n - 1})^a F_a(Z_{i,j}; i \neq j)$.
\end{example}

\begin{lem} \label{lemma-adjunction}
 Suppose $\cX$ is an algebraic stack and $\phi: \cX \to Y$ is an adequate moduli space.  Then for any quasi-coherent $\oh_{Y}$-algebra $\cB$, the adjunction morphism $\cB \to \phi_* \phi^* \cB$ is an adequate homeomorphism.  
 \end{lem}

\begin{proof}  The question is local in the \'etale topology on $Y$ so we may assume that $Y$ is affine.  As $\phi_*$ and $\phi^*$ commute with arbitrary direct sums, the adjunction map $\cB \to \phi_* \phi^* \cB$ is an isomorphism if $\cB$ is a polynomial algebra over $\Gamma(Y, \oh_Y)$.  In general, we can write $\cB$ as a quotient of a polynomial algebra $\cB'$ over $R$ and the statement follows directly from Lemma \ref{lemma-ideal}.
\end{proof}

\begin{remark} With the notation of Remark \ref{remark-git}, Lemma \ref{lemma-adjunction} implies that for an $R^G$-algebra $B$, the adjunction map
$$B \to (B \tensor_{R^G} R)^G$$
is an adequate homeomorphism.  If $S=\Spec(k)$ where $k$ is a field of characteristic $p$ and $G$ is a reductive group, this is \cite[Fact ($1$) on p. 195]{git3}.
\end{remark}

\begin{example}  With the notation of Example \ref{example1}, the quasi-coherent $\oh_{Y}$-algebra $\cB$ associated with $k[y,z]/y$ on $\Spec(k[y,z])$ provides an example where $\phi^* \phi_* \cB \to \cB$ is not surjective.  With the notation of Example \ref{example2}, the quasi-coherent $\oh_Y$-algebra $\cB$ associated with $R^G/(x_1,x_2)$ provides an example where $\phi^* \phi_* \cB \to \cB$ is not injective.
\end{example}

\begin{prop}  \label{proposition-base-change}
Suppose that $\cX$ and $\cX'$ are algebraic stacks and that
$$\xymatrix{
\cX' \ar[d]^{\phi'} \ar[r]^{g'}	 \ar @{} [dr] |{\square}		& \cX \ar[d]^{\phi} \\
Y' \ar[r]^{g}				& Y 
}$$
is a cartesian diagram with $Y$ and $Y'$ algebraic spaces.  Then
\begin{enumerate1}
\item If $g$ is flat and $\phi: \cX \to Y$ is an adequate moduli space, then $\phi': \cX' \to Y'$ is an adequate moduli space.
\item If $g$ is fpqc and $\phi': \cX' \to Y'$ is an adequate moduli space, then $\phi: \cX \to Y$ is an adequate moduli space.
\item If $\phi$ is an adequate moduli space, then $\oh_{Y'} \to \phi_* \oh_{\cX'}$ is an adequate homeomorphism.  The morphism $\phi'$ factors as an adequate moduli space $\cX' \to \sSpec_{Y'} (\phi'_* \oh_{\cX'})$ followed by an adequate homeomorphism  $\sSpec_{Y'} (\phi'_* \oh_{\cX'}) \to Y'$.  
\item If $\cA$ is quasi-coherent $\oh_{\cX}$-algebra, the adjunction morphism $g^* \phi_* \cA \to \phi'_* g'^* \cA$ is an adequate homeomorphism.
\end{enumerate1}
\end{prop}

\begin{proof}
For part (1), Proposition \ref{proposition-adequately-affine}(6) implies that $\phi'$ is adequately affine.  If $g$ is flat, then flat base change implies  $\oh_{Y'} \to \phi'_* \oh_{\cX'}$ is an isomorphism.  For part (2), Proposition \ref{proposition-adequately-affine}(4) implies that $\phi$ is adequately affine and fpqc descent implies that $\oh_Y \to \phi_* \oh_{\cX}$ is an isomorphism.  

\medskip \noindent
For part (3), since $\phi'$ is adequately affine $\cX' \to \sSpec_{Y'} (\phi'_* \oh_{\cX'})$ is an adequate moduli space.  Since the question is local in the fpqc topology, we may assume that $Y' \to Y$ is affine and defined by a quasi-coherent $\oh_Y$-algebra $\cB$.  By Lemma \ref{lemma-adjunction}, $\cB \to \phi_* \phi^* \cB$ is an adequate homeomorphism but this maps corresponds canonically to $g_* \oh_{Y'} \to g_* \phi'_* \oh_{\cX'}$.  

\medskip \noindent
For part (4), the diagram
$$\xymatrix{
\sSpec_{\cX'} (g'^* \cA) \ar[r] \ar[d]  \ar @{} [dr] |{\square}		& \sSpec_{\cX} (\cA) \ar[d] \\
\sSpec_{Y} (g^* \phi_* \cA) \ar[r]		& \sSpec_Y (\phi_* \cA)
}$$
is cartesian so the statement follows from part (3).
\end{proof}

\begin{example}  With the notation of Example \ref{example1}, we have a diagram
$$\xymatrix{
				& [\Spec(k[x]) / \ZZ_p] \ar@{^(->}[r] \ar[d]^{\phi'} \ar[ld]^{\varphi}  \ar @{} [dr] |{\square}		& [\AA^2 / \ZZ_p] \ar[d]^{\phi} \\
\Spec(k[x])	\ar[r]		& \Spec(k[x^p]) \ar@{^(->}[r]^{y=0}			& \Spec(k[y,z])
}$$
where the square is cartesian and $\phi$ and $\varphi$ are adequate moduli spaces.  The base change $\phi'$ is not an adequate moduli space but $\Spec(k[x]) \to \Spec(k[x^p])$ is an adequate homeomorphism.
\end{example}

\begin{lem} \label{lemma-adequate-affine-algebra}  Let $\cX$ be an algebraic stack and let $\phi: \cX \to Y$ be an adequate moduli space.  Let $\cA$ be a quasi-coherent sheaf of $\oh_{\cX}$-algebras.  Then $\sSpec_{\cX} (\cA)  \to \sSpec_Y (\phi_* \cA)$ is an adequate moduli space.  In particular, if $\cZ \subseteq \cX$ is a closed substack, then $\cZ \to Y':=\sSpec (\phi_* \oh_{\cZ})$ is an adequate moduli space.  The induced morphism $Y' \to \im \cZ$ to the scheme-theoretic image of $\cZ$ in $Y$ is an adequate homeomorphism.
\end{lem}

\begin{proof}  Since $\sSpec_{\cX} (\cA) \to Y$ is adequately affine, it follows that $\sSpec_{\cX} (\cA)  \to \sSpec_Y (\phi_* \cA)$ is an adequate moduli space.  The final statement follows directly from Lemma \ref{lemma-ideal}.
\end{proof}

\begin{lem} \label{lemma-nagata} (Analogue of Nagata's fundamental lemmas)
Let $\phi: \cX \to Y$ be an adequately affine morphism.  Then:
\begin{enumerate1}
\item For any quasi-coherent sheaf of ideals $\cI$ on $\cX$, the inclusion
$$\phi_* \oh_{\cX}  / \phi_* \cI \to \phi_* (\oh_{\cX} / \cI)$$
is an adequate homeomorphism.
\item For any pair of quasi-coherent sheaves of ideals $\cI_1, \cI_2$ on $\cX$, the inclusion
$\phi_* \cI_1 + \phi_* \cI_2 \to \phi_*(\cI_1 + \cI_2)$
induces an adequate homeomorphism
$$\oh_Y / (\phi_* \cI_1 + \phi_* \cI_2) \to \oh_Y / \phi_*(\cI_1 + \cI_2)$$
In other words, for every section $s \in \Gamma(\Spec(A) \to Y, \phi_* (\cI_1 + \cI_2))$, there exists $N > 0$ such that $s^N \in \Gamma(\Spec(A) \to Y, \phi_*\cI_1 + \phi_* \cI_2)$.
\end{enumerate1}
\end{lem}

\begin{proof}
Part (1) is obvious. For part (2), we may assume that $Y$ is affine.  The exact sequence 
$$0 \to \cI_1 \to \cI_1 + \cI_2 \to \cI_2 / \cI_1 \cap \cI_2 \to 0$$ 
induces a commutative diagram
$$\xymatrix{
		&				& \Gamma(\cX, \cI_2)\ar[d] \ar[rd] \\
0 \ar[r]	& \Gamma(\cX, \cI_1) \ar[r]	& \Gamma(\cX, \cI_1 + \cI_2) \ar[r]	& \Gamma(\cX, \cI_2 / \cI_1 \cap \cI_2)
}$$
where the bottom row is left exact.  Let $s \in \Gamma(\cX, \cI_1+ \cI_2)$ with image $\bar s$ in $\Gamma(\cX, \cI_2 / \cI_1 \cap \cI_2)$.  Since $\phi$ is adequately affine, there exist $N> 0$ and $t_2 \in \Gamma(\cX, \cI_2)$ such that $t_2 \mapsto \bar{s}^N$.  It follows that $t_2 - s^N \in \Gamma(\cX, \cI_2)$.
\end{proof}

\begin{remark} 
Part (2) above implies that for any set of quasi-coherent sheaves of ideals $\cI_{\alpha}$ 
$$\oh_Y / \big( \sum_{\alpha} \phi_* \cI_{\alpha} \big) \to \oh_Y / \big( \phi_* ( \sum_{\alpha} \cI_{\alpha} ) \big)$$
is an adequate homeomorphism.
\end{remark}

\begin{remark} With the notation of Remark \ref{remark-git}, property (1) translates into the statement that the natural inclusion $A^G/ (I \cap A^G) \hookrightarrow (A/I)^G$ is universally adequate for any invariant ideal $I \subseteq A$.  Property (2) translates into the statement that for any pair of invariant ideals $I_1, I_2 \subseteq A$, the induced inclusion  $(I_1 \cap A^G) + (I_2 \cap A^G) \hookrightarrow (I_1 + I_2) \cap A^G$ has the property that for any $s \in (I_1 + I_2) \cap A^G$, there exists $N > 0$ such that $s^N \in (I_1 \cap A^G) + (I_2 \cap A^G)$.  Note that if $S$ is defined over $\FF_p$, then by Lemma \ref{lemma-adequate-charp}, the integer $N$ can be chosen to be a prime power. If $S=\Spec(k)$ where $k$ is a field of characteristic $p$ and $G$ is a reductive group, this is \cite[Lemma 5.1.B and 5.2.B]{nagata_invariants-affine} and \cite[Lemma A.1.2 and Fact ($2$), p.195]{git3}.  
\end{remark}

\begin{example} Example \ref{example1} illustrates that the map in part (1) is not always surjective.  For an example where the map in part (2) is not an isomorphism, consider the dual action of $G=\ZZ_2$ on $A=\FF_2[x_1, x_2, y]/(x_1^2, x_2^2)$ given by $(x_1, x_2, y) \mapsto (x_1, x_2, y + x_1 + x_2)$.  Then the inclusion $(x_1) \cap A^G + (x_2) \cap A^G \hookarr (x_1,x_2) \cap A^G$ is not surjective as $y(x_1+x_2) \in (x_1, x_2) \cap A^G$ is not in the image.
\end{example}

\subsection{Geometric properties}

\begin{thm} \label{theorem-adequate}
Let $\cX$ be an algebraic stack and let $\phi: \cX \to Y$ be an adequate moduli space.  Then
\begin{enumerate1}
\item \label{surjective} 
	The morphism $\phi$ is surjective.
\item \label{closed}
	The morphism $\phi$ is universally closed.
\item \label{submersive}
	The morphism $\phi$ is universally submersive.
\item \label{intersection}
 	If $Z_1, Z_2$ are closed substacks of $\cX$, then
$$  \im Z_1 \cap \im Z_2 = \im (Z_1 \cap Z_2) $$
where the intersections and images are set-theoretic.
\item \label{points}
For an algebraically closed field $k$, there is an equivalence relation defined on $[\cX(k)]$ by $x_1 \sim x_2 \in [\cX(k)]$ if $\overline{ \{x_1\}} \cap \overline{ \{x_2\} } \ne \emptyset$ in $\cX \times_{\ZZ} k$ inducing a bijective map $[\cX(k)]/ \simm \to Y(k)$.  That is, $k$-valued points of $Y$ are $k$-valued points of $\cX$ up to orbit closure equivalence.
\end{enumerate1}
\end{thm}

\begin{proof}  For part (\ref{surjective}), if $\Spec(k) \to Y$ is an arbitrary map from a field $k$, then Proposition \ref{proposition-base-change}(3) implies that $\cX \times_Y \Spec(k)$ is non-empty.  For part (\ref{closed}), if $\cZ \subseteq \cX$ is a closed substack, then Lemma \ref{lemma-adequate-affine-algebra} implies that $\cZ \to Y' = \sSpec_Y (\phi_* \cZ)$ is an adequate moduli space and $Y' \to \im \cZ$ is an adequate homeomorphism.  Using part (1), it follows that the composition $\cZ \to Y' \to \im \cZ$ is surjective so that $\phi(\cZ)$ is closed.  Proposition \ref{proposition-base-change}(3) then implies that $\phi$ is universally closed.   
Part (\ref{submersive}) follows from parts (\ref{surjective}) and (\ref{closed}).
Part (\ref{intersection}) follows from Lemma \ref{lemma-nagata}(2).  Part (\ref{points}) follows from (\ref{intersection}) as in the argument of \cite[Theorem 4.16(iv)]{alper-good}. 
\end{proof}

\subsection{Preservation of properties}

\begin{prop}  Let $\cP \in \{$reduced, connected, irreducible, normal$\}$ be a property of algebraic stacks.  Let $\cX$ be an algebraic stack and let $\phi: \cX \to Y$ be an adequate moduli space.  If $\cX$ has property $\cP$, then so does $Y$.
\end{prop}

\begin{proof}
The first three are clear.  For $\cP = $ ``normality'', we may assume that $Y$ is affine and integral and the statement follows since $\phi$ is universal for maps to affine schemes.  
\end{proof}

\subsection{Flatness}
If $\phi: \cX \to Y$ is a good moduli space with both $\cX$ and $Y$ defined over a base $S$ and $\cF$ is quasi-coherent $\oh_{\cX}$-module flat over $S$, then $\phi_* \cF$ is also flat over $S$ (see \cite[Theorem 4.16(ix)]{alper-good}). The following example shows that the corresponding property does not hold for adequate moduli spaces.

\begin{example}
Let $R = \FF_2[x,y]/xy$ and consider the dual action of $G=\ZZ/ 2 \ZZ$ on $A=R[z,w]$ given by $(z,w) \mapsto(z+x, w+y)$.  Then $R \to A$ flat but we claim that $R \to A^G$ is not flat.  Indeed, since the annihilator of $y$ in $R$ is the ideal $(x)$, we have the injection $R/x \stackrel{y}{\hookarr} R$.  But $A^G / x \to A^G$ is not injective.  The element $f = xw \in A^G$ satisfies $fy = 0$ but is not divisible by $x$.  
\end{example}

\subsection{Vector bundles}  If $\phi: \cX \to Y$ is a good moduli space with $\cX$ locally Noetherian and $\cF$ is a vector bundle on $\cX$ such that at all points $x: \Spec(k) \to \cX$ with closed image, the $G_x$-representation $\cF \tensor k$ is trivial, then $\cF$ is the pullback of a vector bundle on $Y$ (see \cite[Theorem 10.3]{alper-good}).  This is not true for adequate moduli spaces.

\begin{example}
Suppose $\text{char}(k) = p$.  Let $S = k[\epsilon]/(\epsilon^2)$ be the dual numbers.  Consider the group scheme  $\alpha_{p,S} = \Spec(k[x, \epsilon]/(\epsilon^2,x^p))$.  Then $B \alpha_{p,S} \to S$ is an adequate moduli space (see Theorem \ref{theorem-finite-group}). Trivial representations of $\alpha_p$ over the closed point have non-trivial deformations.  Consider the line bundle $\cL$ on $B \alpha_{p,S}$ corresponding to the character
$$\begin{aligned}
 \alpha_p  &\to \GG_{m,S} =  \Spec(k[\epsilon, t]_t/(\epsilon^2))  \\
				1 + \epsilon x			& \mapsfrom  t
\end{aligned}$$
This restricts to the trivial line bundle under the closed immersion $B \alpha_{p,k} \hookarr B \alpha_{p,S}$ but is not the pullback of a line bundle on $S$.  One can construct similar examples for $\ZZ / p \ZZ$.
\end{example}

%%%%
\section{Finiteness results}

\subsection{Historical context} \label{section-finiteness}

In this section, we show that if $\cX \to Y$ is an adequate moduli space defined over a Noetherian algebraic space $S$, then $\cX \to S$ of finite type implies that $Y \to S$ is of finite type.  This can be considered as a generalization of Nagata's result that if $G$ is a geometrically reductive group over $k$ and $A$ is a finitely generated $k$-algebra, then $A^G$ is finitely generated over $k$ (see \cite{nagata_invariants-affine} or \cite[Appendix 1.C]{git3}).  See \cite[Section 3.6]{newstead} for a more complete discussion on the finite generation of the invariant rings.  Theorem \ref{theorem-finiteness} generalizes Seshadri's result \cite[Theorem 2]{seshadri_reductivity} for actions by reductive group schemes $G \to \Spec(R)$ where $R$ is universally Japanese as well as Borsari and Ferrer Santos's result \cite[Theorem 4.3]{ferrer-santos} on actions by geometrically reductive commutative Hopf algebras over fields.  The theorem here also generalizes \cite[Theorem 4.16(xi)]{alper-good} where the analogous statement is proved for good moduli spaces over an excellent base.  In \cite{alper-dejong}, a categorical framework for the adequacy condition is considered and the main result simultaneously generalizes Theorem \ref{theorem-finiteness} and the finiteness results in \cite{ferrer-santos} for actions of geometrically reductive non-commutative Hopf algebras.

\subsection{General result about finite generation of subrings}
We will apply the following result, which was discovered jointly with Johan de Jong.

\begin{thm} 
\label{theorem-finite-type}
Consider a commutative diagram of schemes
$$
\xymatrix{
X \ar[d] \ar[rd] \\
Y \ar[r]	& S
}
$$
Assume that:
\begin{enumeratea}
\item The schemes $Y$ and $S$ are Noetherian.
\item The morphism $X \to S$ is of finite type.
\item The morphism $X \to Y$ is quasi-compact and universally submersive.
\end{enumeratea}
Then $Y \to S$ is of finite type.
\end{thm}

\begin{remark}  This theorem generalizes several known special cases:
\begin{enumerate}
\item If $X \to Y$ is faithfully flat, this is \cite[IV.2.7.1]{ega}.  
(This is true even without the Noetherian hypothesis.)
\item If $X \to Y$ is pure, 
this is \cite[Theorem 1]{hashimoto_pure}.  (Here $Y$ does not 
need to be assumed Noetherian as it is immediately implied by the 
Noetherianness of $X$ and purity.)
\item If $X \to Y$ is surjective and universally open, 
$Y$ is reduced and $S$ is a universally catenary Nagata scheme, 
this is \cite[Theorem 2.3]{hashimoto_geometric}.  
(This is true without assuming $Y$ is Noetherian.)
\item If $X \to Y$ is surjective and proper and $S$ is excellent, 
this is \cite[Theorem 4.2]{hashimoto_geometric}.
\end{enumerate}

\end{remark}

\begin{proof}
We may assume that $S = \text{Spec}(R)$ and $Y = \text{Spec}(B)$ are affine.  
Furthermore, since a Noetherian scheme is of finite type over a ring $R$ if and 
only if the reduced subschemes of the irreducible components are finite type 
over $R$, we may assume that $Y$ is integral (\cite[p. 169]{fogarty}).

\medskip
\noindent
The morphism $X \to \text{Spec}(B)$ is flat over a non-empty
open subscheme $U \subset \text{Spec}(B)$. 
By \cite[Theorem 5.2.2]{raynaud-gruson}, there
exists a $U$-admissible blowup
$$
b : \tilde Y \longrightarrow Y = \text{Spec}(B)
$$
such that the strict transform $X'$ of $X$ is flat over $\tilde Y$.
For every point $y \in \tilde Y$ we can find a discrete valuation
ring $V$ and morphism $\text{Spec}(V) \to \tilde Y$ whose generic
point maps into $U$ and whose special point maps to $y$. By assumption
there exist a local map of discrete valuation rings $V \to V'$
and a commutative diagram
$$
\xymatrix{
X \ar[d] & & \text{Spec}(V') \ar[ll] \ar[d] \\
Y & \tilde Y \ar[l] & \text{Spec}(V) \ar[l]
}
$$
By definition of the strict transform we see that the product map
$\text{Spec}(V') \to \tilde Y \times_Y X$ maps into the strict transform.
Hence we conclude there exists a point on $X'$ which maps to $y$, that is,
we see that $X' \to \tilde Y$ is surjective. 
By \cite[IV.2.7.1]{ega}, 
 we conclude that $\tilde Y \to \text{Spec}(R)$
is of finite type.

\medskip\noindent
Let $I \subset B$ be an ideal such that $\tilde Y$ is the blowup of
$\text{Spec}(B)$ in $I$. Choose generators $f_i \in I$, $i = 1, \ldots, n$.
For each $I$ the affine ring
$$
B_i = B[f_j/f_i;\ j = 1, \ldots, \hat i, \ldots n] \subset f.f.(B)
$$
in the blowup is of finite type over $R$. Write
$B = \text{colim}_{\lambda \in \Lambda} B_\lambda$
as the union of its finitely generated $R$-subalgebras.
After shrinking $\Lambda$ we may assume that each $B_\lambda$ contains
$f_i$ for all $i$. Set $I_\lambda = \sum f_iB_\lambda \subset B_\lambda$
and let
$$
B_{\lambda, i} = B_\lambda[f_j/f_i;\ j = 1, \ldots, \hat i, \ldots n]
\subset f.f.(B_\lambda) \subset f.f.(B)
$$
After shrinking $\Lambda$ we may assume that the canonical maps
$B_{\lambda, i} \to B_i$ are surjective for each $i$ (as $B_i$ is finitely
generated over $R$). Hence for such a $\lambda$ we have
$B_{\lambda, i} = B_i$! So for such a $\lambda$ the blowup of
$\text{Spec}(B_\lambda)$ in $I_\lambda$ is {\bf equal} to the blowup
of $\text{Spec}(B)$ in $I$. Set $Y_{\lambda} = \text{Spec}(B_{\lambda})$.
 Thus the composition
$$
\tilde Y \longrightarrow Y \longrightarrow Y_{\lambda}
$$
is a projective morphism and we see that
$$
(Y \to Y_{\lambda})_*\mathcal{O}_Y \subset
(\tilde Y \to Y_{\lambda})_*\mathcal{O}_{\tilde Y}
$$
and the last sheaf is a coherent $\mathcal{O}_{Y_{\lambda}}$-module (\cite[III.3.2.1]{ega}).
Hence $(Y \to Y_{\lambda})_*\mathcal{O}_Y$ is also coherent so that $Y \to Y_{\lambda}$
is finite and we win.
\end{proof}

\medskip \noindent
Let $\phi: \cX \to Y$ be an adequate moduli space where $\cX$ is an algebraic stack of finite type over a Noetherian base $S$.  If we knew a priori that $Y$ is Noetherian, then the above theorem would immediately imply that $Y \to S$ is of finite type by using property Theorem \ref{theorem-adequate}(\ref{submersive}).  However, it is not true in general that if $\cX$ is Noetherian then $Y$ is Noetherian.

\begin{example} We quickly recall Nagata's example (see \cite{nagata_actions} and \cite[Example 6.5.1]{kollar_quotients}) of a Noetherian affine scheme $\Spec(R)$ which is defined but not of finite type over $\FF_p$ such that $R^{\ZZ/p \ZZ}$ is not Noetherian.
Let $K = \FF_p(x_1, x_2, \ldots)$.  Let $D:= \sum_i x_i^{p+1} \frac{\partial}{ \partial x_i}$ be a derivation of $K$.   Then $R = K[\epsilon]/ (\epsilon^2)$ is a local Artin ring (and thus Noetherian).  There is a dual action of $\ZZ / p \ZZ$ on $R$ given on a generator by $f + \epsilon g \mapsto f + \epsilon(g + D(f))$.  One can show that the ring of invariants $R^{\ZZ / p \ZZ} = F + \epsilon K$ is non-Noetherian, where $F = \{f \in K \mid D(f) = 0\}$.
\end{example}

\subsection{The main finiteness result} 
The proof of Theorem \ref{theorem-finiteness} will be by Noetherian induction.  Consider the following property of a Noetherian algebraic stack $\cX$ defined over a Noetherian ring $R$.
\begin{enumerate}
\item[($\star$)] The ring $\Gamma(\cX, \oh_{\cX})$ is a finite type $R$-algebra and for every coherent $\oh_{\cX}$-module $\cF$, the $\Gamma(\cX, \oh_{\cX})$-module $\Gamma(\cX, \cF)$ is of finite type.
\end{enumerate}

\begin{lem} \label{lemma-finiteness1}
Let $\cX$ be an adequate algebraic stack of finite type over a Noetherian ring $R$.  Let $\cI$ be a coherent sheaf of ideals in $\oh_{\cX}$ such that $\Gamma(\cX, \oh_{\cX} / \cI)$ is a finite type $R$-algebra.  Then $\Gamma(\cX, \oh_{\cX}/\cI)$ is a finite type $\Gamma(\cX, \oh_{\cX})$-module and $\im(\Gamma(\cX, \oh_{\cX}) \to \Gamma(\cX, \oh_{\cX}/\cI))$ is a finite type $R$-algebra.
\end{lem}

\begin{proof} Since $\cX$ is adequate, $\Gamma(\cX, \oh_{\cX}) \to \Gamma(\cX, \oh_{\cX}/\cI)$ is adequate and, in particular, integral.  Since $\Gamma(\cX, \oh_{\cX} / \cI)$ is a finite type $R$-algebra, the subalgebra $\im(\Gamma(\cX, \oh_{\cX}) \to \Gamma(\cX, \oh_{\cX}/\cI))$ also is.
\end{proof}

\begin{lem} \label{lemma-finiteness2}
Let $\cX$ be an adequate algebraic stack finite type over a Noetherian ring $R$.  Suppose that $\cI$ and $\cJ$ are quasi-coherent sheaves of ideals in $\oh_{\cX}$ such that $\cI \cJ = 0$.  If $(\star)$ holds for the closed substacks defined by $\cI$ and $\cJ$, then $(\star)$ holds for $\cX$.
\end{lem}

\begin{proof}
By Lemma \ref{lemma-finiteness1}, there exists a finite type $R$-subalgebra $B \subseteq \Gamma(\cX, \oh_{\cX})$ such that $B \to \im (\Gamma(\cX, \oh_{\cX}) \to \Gamma(\cX, \oh_{\cX} / \cI) )$ and $B \to \im (\Gamma(\cX, \oh_{\cX}) \to \Gamma(\cX, \oh_{\cX} / \cJ) )$ are surjective.   Since $(\star)$ holds for the closed substack defined by $\cJ$ and $\cI$ is an $\oh_{\cX}/\cJ$-module, we may choose generators $x_1, \ldots, x_n$ of  $\Gamma(\cX, \cI)$ as an $\Gamma(\cX, \oh_{\cX}/\cJ)$-module.  We claim that $B[x_1, \ldots, x_n] \to \Gamma(\cX, \oh_{\cX})$ is surjective.  Let $f \in \Gamma(\cX, \oh_{\cX})$.  There exists $g \in B$ such that $f$ and $g$ have the same image in $\Gamma(\cX, \oh_{\cX}/\cI)$ so we may assume $f \in \Gamma(\cX, \cI)$.  We can write $f = a_1 x_1 + \cdots + a_n x_n$ with $a_i \in \Gamma(\cX, \oh_{\cX})$.  But there exists $a'_i \in B$ such that $a_i$ and $a'_i$ have the same image in $\Gamma(\cX, \oh_{\cX}/\cJ)$ so $f = a'_1 x_1 + \cdots + a'_n x_n$ is in the image of $B[x_1, \ldots, x_n] \to \Gamma(\cX, \oh_{\cX})$.  Therefore, $\Gamma(\cX, \oh_{\cX})$ is a finite type $R$-algebra.

\medskip \noindent
Let $\cF$ be a coherent $\oh_{\cX}$-module.  Consider the exact sequence
$$ 0 \to \Gamma(\cX, \cI \cF) \to \Gamma(\cX, \cF) \to \Gamma(\cX, \cF/ \cI\cF)$$
Now $\cI \cF$ is a $\oh_{\cX}/\cJ$-module and $\cF / \cI \cF$ is a $\oh_{\cX}/\cI$-module so by the hypotheses both $\Gamma(\cX, \cI \cF)$ and $\Gamma(\cX, \cF / \cI \cF)$ are finite type $\Gamma(\cX, \oh_{\cX})$-modules.  It follows that $\Gamma(\cX, \cF)$ is a finite type $\Gamma(\cX, \oh_{\cX})$-module.   
\end{proof}

\begin{thm} \label{theorem-finiteness}
Let $\cX$ be a finite type algebraic stack over a locally Noetherian algebraic space $S$.  Let $\phi: \cX \to Y$ be an adequate moduli space where $Y$ is an algebraic space over $S$.  Then $Y \to S$ is of finite type and for every coherent $\oh_{\cX}$-module $\cF$, the $\oh_{Y}$-module $\phi_* \cF$ is coherent.
\end{thm}

\begin{proof}  We may assume that $S = \Spec(R)$.  By Noetherian induction, we may assume that $(\star)$ holds for any closed substack $\cZ \subseteq \cX$ defined by a non-zero sheaf of ideals.  For $f \neq 0 \in \Gamma(\cX, \oh_{\cX})$, if $\ker(\oh_{\cX} \xrightarrow{f} \oh_{\cX})$ is non-zero, then by applying Lemma \ref{lemma-finiteness2} with the ideals sheaves $(f)$ and $\ker(\oh_{\cX} \xrightarrow{f} \oh_{\cX})$, we see that $(\star)$ holds for $\cX$.  Therefore, we may assume that every $f \neq 0 \in \Gamma(\cX, \oh_{\cX})$ is a non-zero divisor; that is, $\Gamma(\cX, \oh_{\cX})$ is an integral domain. 

\medskip \noindent Let $I \subseteq \Gamma(\cX, \oh_{\cX})$ be an ideal and let $f \neq 0 \in I$.  Since $f$ is a non-zero divisor, we have an exact sequence 
$$ 0 \to \Gamma(\cX, \oh_{\cX}) \xrightarrow{f} \Gamma(\cX, \oh_{\cX}) \to \im(\Gamma(\cX, \oh_{\cX}) \to \Gamma(\cX, \oh_{\cX}/ (f)) ) \to 0$$
By the induction hypothesis and Lemma \ref{lemma-finiteness1}, $\Gamma(\cX, \oh_{\cX}) / (f)$ is a finite type $R$-algebra.  The image of $I$ in $\Gamma(\cX, \oh_{\cX}) / (f)$ is a finitely generated ideal.  Therefore, $I$ is finitely generated and $\Gamma(\cX, \oh_{\cX})$ is Noetherian.  

\medskip \noindent If $U \to \cX$ is a smooth presentation, then the composition $U \to \cX \to Y$ is universally submersive by Theorem \ref{theorem-adequate}(\ref{submersive}).  It follows from Theorem \ref{theorem-finite-type} that $Y \to \Spec(R)$ is of finite type.  

\medskip \noindent
Let $\cF$ be a coherent $\oh_{\cX}$-module.  We wish to show that $\Gamma(\cX, \cF)$ is a finite type $\Gamma(\cX, \oh_{\cX})$-module.  By Noetherian induction again, we may assume that for every proper quotient $\cF \mapsonto \cF'$, the $\Gamma(\cX, \oh_{\cX})$-module $\Gamma(\cX, \cF')$ is of finite type. The statement is true if $\Gamma(\cX, \cF) = 0$; otherwise, let $s \neq 0 \in \Gamma(\cX, \cF)$.   Denote by $s \cdot \cF$ the image of $s: \oh_{\cX} \to \cF$ so that $\s \cdot \cF \cong \oh_{\cX} / \cI$ where $\cI = \ker(s: \oh_{\cX} \to \cF)$.   Consider the exact sequence
$$0 \to \Gamma(\cX, s \cdot \cF) \to \Gamma(\cX, \cF) \to \Gamma(\cX, \cF / s \cdot \cF)$$
By the induction hypothesis, $\Gamma(\cX, \cF / s \cdot \cF)$ is a finite type $\Gamma(\cX, \oh_{\cX})$-module.  If $\cI = 0$, then $s \cdot \cF = \oh_{\cX}$ and as $\Gamma(\cX, \oh_{\cX})$ is Noetherian, $\Gamma(\cX, \cF)$ is also a finite type $\Gamma(\cX, \oh_{\cX})$-module.  If $\cI \neq 0$, then by the inductive hypothesis and Lemma \ref{lemma-finiteness1}, $\Gamma(\cX, s \cdot \cF)$ is a finite type $\Gamma(\cX, \oh_{\cX})$-module so that $\Gamma(\cX, \cF)$ is also.
\end{proof}

%%%%
\section{Uniqueness of adequate moduli spaces}

\medskip \noindent
In this section, we show that if $\phi: \cX \to Y$ is an adequate moduli space, then $\phi$ is universal for maps to algebraic spaces which are locally separated or Zariski-locally have affine diagonal; that is, for any other morphism $\psi: \cX \to Z$ to an algebraic space $Z$ which is either locally separated or Zariski-locally has affine diagonal, there exists a unique morphism $\chi: Y \to Z$ such that $\psi = \chi \circ \phi$.  We believe that an adequate moduli space should be universal for maps to arbitrary algebraic spaces.

\medskip \noindent
\subsection{General result}
It follows from general methods that an adequate moduli space $\phi: \cX \to Y$ is universal for maps to locally separated algebraic spaces.  The following technique was used by David Rydh in \cite{rydh_quotients} to show that geometric quotients are universal for such maps.

\medskip \noindent
If an algebraic stack $\cX$ admits an adequate moduli space, then the relation that $x \sim_c y \in \cX(k)$ if $\overline{ \{x \}} \cap \overline{ \{ y \} } \neq \emptyset $ in $| \cX \times_{\ZZ} k|$ defines an equivalence relation (see \ref{theorem-adequate}(\ref{points})).  This is not true for an arbitrary stack; consider $[\PP^1 / \GG_m]$.  However, by using chains of orbit closures, we can define an equivalence relation as follows:  two geometric points $x,y \in \cX(k)$, are said to be \emph{closure equivalent}  (denoted $x \sim_c y$) if there is a sequence of points $x=x_1, x_2, \cdots, x_{n-1}, x_n = y \in \cX(k)$ such that for $i = 1, \cdots n-1$, $\overline{ \{ x_i \}} \cap \overline { \{x_{i+1} \} } \neq \emptyset$ in $|\cX \times _{\ZZ} k|$.

\begin{prop}  \label{proposition-locally-separated}
Let $\cX$ be an algebraic stack and let $Y$ be an algebraic space.  Suppose that $\phi: \cX \to Y$ is a morphism such that 
\begin{enumeratea}
\item The map $[\cX(k)] / \simmc \to Y(k)$ is bijective for all algebraically closed $\oh_S$-fields $k$. 
\item The morphism $\phi$ is universally submersive.
\item The morphism $\oh_Y \to \phi_* \oh_{\cX}$ is an isomorphism.
\end{enumeratea}
Then $\phi$ is universal for maps to locally separated algebraic spaces.
\end{prop}

\begin{remark} Condition (a) says that $Y$ has the right points, condition (b) says that $Y$ has the right topology and condition (c) says that $Y$ has the right functions.   Conditions (a) and (b) are stable under arbitrary base change while condition (c) is stable under flat base change.  Conditions (a)$-$(c) descend in the fpqc topology.
\end{remark}

\begin{proof} We need to show that for any locally separated algebraic space $Z$
$$\Hom(Y, Z) \to \Hom(\cX,Z)$$ is bijective.  The injectivity is straightforward (see
\cite[Proposition 7.2]{rydh_submersions}). Consider a morphism $\psi: \cX \to Z$ where $Z$ is a
locally separated algebraic space.  Since $\cX \to Y$ is universally submersive, it follows from
\cite[Theorem 7.4]{rydh_submersions} that
$$\Hom(Y, Z) \to \Hom(\cX, Z) \rrarrows \Hom((\cX \times_Y \cX)_{\red}, Z)$$ is exact.  Therefore,
it suffices to show that $\psi \circ p_1 = \psi \circ p_2$ where $p_1$ and $p_2$ are the projections
$(\cX \times_Y \cX)_{\red} \to \cX$.  We note that $\psi \circ p_1 = \psi \circ p_2$ if and only if
there exists a $\Lambda: (\cX \times_Y \cX)_{\red} \to \cX \times_Z \cX$ such that
$$\xymatrix{
(\cX \times_Y \cX)_{\red} \ar[r]^{\Lambda} \ar[d]	& \cX \times_Z \cX \ar[dl] \\
\cX \times \cX
}$$
commutes.  Consider the cartesian diagram
$$\xymatrix{
\cW \ar[r] \ar[d]	 \ar @{} [dr] |{\square}				& (\cX \times_Y \cX)_{\red} \ar[d] \\
\cX \times_Z \cX \ar[r]\ar[d]	 \ar @{} [dr] |{\square}		& \cX \times \cX \ar[d] \\
Z  \ar[r]					& Z \times Z
}$$
The monomorphism $\cW \to (\cX \times_Y \cX)_{\red}$ is surjective by property (1) and also an immersion since $Z$ is locally separated.  It follows that $\cW \to (\cX \times_Y \cX)_{\red}$ is an isomorphism and that $\psi \circ p_1 = \psi \circ p_2$.
\end{proof}

\begin{remark}
It is not true that the conditions (a) $-$ (c) imply that $\phi$ is universal for maps to arbitrary algebraic spaces.  Indeed, if $X$ is the non-locally separated affine line (that is, the bug-eyed cover), then $X \to \AA^1$ satisfies conditions (a) $-$ (c) but is not an isomorphism. 
\end{remark}

\subsection{Universality for adequate moduli spaces}

\begin{thm} \label{theorem-uniqueness}
Let $\cX$ be an algebraic stack and let $\phi: \cX \to Y$ be an adequate moduli space.  Then $\phi$ is universal for maps to algebraic spaces which are either locally separated or Zariski-locally have affine diagonal. 
\end{thm}

\begin{proof} Let $Z$ be an algebraic space.  We need to show that the natural map
$$\Hom(Y, Z) \to \Hom(\cX, Z)$$
is bijective.  The injectivity of the map is straightforward.  Proposition \ref{proposition-locally-separated} shows that it is surjective if $Z$ is locally separated.
  Let $\psi: \cX \to Z$ be a morphism where $Z$ is an algebraic space which Zariski-locally has affine diagonal.  The argument of \cite[Remark 0.5]{git} (see also \cite[Theorem 4.16(vi)]{alper-good}) shows that the question is Zariski-local on $Z$; in particular, the statement holds when $Z$ is a scheme.
 Therefore we may assume that $Z$ is quasi-compact and has affine diagonal.  The question is also \'etale local on $Y$ so we may assume $Y = \Spec(A)$ is an affine scheme.  Furthermore, by replacing $Z$ with $\sSpec_Z (\psi_* \oh_{\cX})$, we may assume that $\oh_Z \to \psi_* \oh_{\cX}$ is an isomorphism.  Since $Y$ is affine, there exists a unique morphism $\eta: Z \to Y$ such that $\phi = \eta \circ \phi$.  
  
  \medskip \noindent
  Since $Z$ has affine diagonal, $\psi: \cX \to Z$ is an adequate moduli space (see Lemma \ref{lemma-propertyP}).  Let $W \to Z$ be a finite surjective map from a scheme $W$ (\cite[Theorem B]{rydh_noetherian}).  Therefore, by Proposition \ref{proposition-base-change} there exists a diagram
$$\xymatrix{
		& \cX \times_Z W \ar[d] \ar[r] \ar[ld]  \ar @{} [dr] |{\square}			& \cX \ar[d]^{\psi} \ar[rd]^{\phi} \\
W' \ar[r]	& W \ar[r]	& Z \ar[r]^{\eta}	&Y
}$$
where $ \cX \times_Z W \to W'$ is an adequate moduli space and $W' \to W$ is an adequate homeomorphism (and in particular integral and surjective).  Since $\cX \times_Z W$ is adequately affine, $\cX \times_Z W \to \Spec(\Gamma(\cX \times_Z W, \oh_{\cX \times_Z W}))$ is also an adequate moduli space. But since $W'$ is a scheme and since we know adequate moduli spaces are universal for maps to schemes, it follows that $W'$ is affine.  The composition $W' \to W \to Z$ is integral and surjective.  It follows from Chevalley's criterion (\cite[Theorem 8.1]{rydh_noetherian}) that $Z$ is affine and $Z \to Y$ is an isomorphism.
\end{proof}

%%%%
\section{Coarse moduli spaces}

\medskip \noindent Recall that if $\cX$ is an algebraic stack, a morphism $\phi: \cX \to Y$ to an algebraic space $Y$ is a \emph{coarse moduli space} if 
\begin{enumerate1}
\item for any algebraically closed field $k$, the map $[\cX(k)]/\simm \to Y(k)$ from isomorphism classes of $k$-valued points of $\cX$ to $k$-valued points of $Y$ is bijective, and
\item the morphism $\phi$ is universal for maps to algebraic spaces; that is, for any morphism $\xi: \cX \to Z$ to an algebraic space $Z$, there exists a unique map $\chi: Y \to Z$ such that $\xi = \chi \circ \phi$.
\end{enumerate1}

\subsection{Keel$-$Mori}

\begin{thm} \label{theorem-keel-mori}
 {\rm (\cite{keel-mori}, \cite{conrad}, \cite{rydh_quotients}) }
Suppose that $\cX$ is an algebraic stack with finite inertia $I_{\cX} \to \cX$.  Then there exists a coarse moduli space $\phi: \cX \to Y$ such that:
\begin{enumerate1}
\item The morphism $\phi$ is separated.
\item If $\cX$ is locally of finite type over a locally Noetherian algebraic space $S$, then $Y \to S$ is locally of finite type. 
\end{enumerate1}
\end{thm}

\medskip \noindent %%
In \cite{keel-mori}, the theorem was proved when $\cX$ was locally of finite presentation over a locally Noetherian scheme $S$.  The Noetherian hypothesis of $S$ was removed in \cite{conrad}.  The finiteness assumptions of $\cX$ were removed in \cite{rydh_quotients}.

\medskip \noindent
We also recall the following proposition which follows from the proof of the Keel$-$Mori theorem in \cite{keel-mori}.  For the generality stated below, we need \cite[Theorem 7.13]{rydh_quotients}.

\begin{prop}  \label{proposition-finite-cover}
Let $\cX$ be a quasi-compact algebraic stack with finite inertia $I_{\cX} \to \cX$ and $\phi: \cX \to Y$ be its coarse moduli space.  Then there exists an \'etale surjective morphism $Y' \to Y$ such that $\cX \times_Y Y'$ admits a finite, flat, finitely presented morphism from an affine scheme.
\end{prop}

\subsection{Keel$-$Mori coarse moduli spaces are adequate}

\medskip \noindent

\begin{prop} 
\label{proposition-coarse}
 Suppose that $\cX$ is an algebraic stack with finite inertia $I_{\cX} \to \cX$.  Let $\phi: \cX \to Y$ be its coarse moduli space.  Then $\phi: \cX \to Y$ is an adequate moduli space.
\end{prop}

\begin{proof}  By Proposition \ref{proposition-finite-cover}  and Proposition \ref{proposition-base-change}, it suffices to assume that there exists a finite, flat morphism $p: U = \Spec(C) \to \cX$.  We may assume that $p$ is locally free of rank $N$.  Let $s, t: R = \Spec(D) \rrarrows U$ be the groupoid presentation.  If $C^R = \Eq(C \rrarrows D)$, then $\phi: \cX \to Y = \Spec(C)^R$ is the coarse moduli space.

\medskip \noindent
Let $\alpha: \cA \to \cB$ be a surjective morphism of quasi-coherent $\oh_{\cX}$-algebras.  Then $\cA$ and $\cB$ correspond to a $C$-algebra $A$ and $B$, and isomorphisms $\beta_A: A \tensor_{C, s} D \iso A \tensor_{C,t} D$ and  $\beta_B: B \tensor_{C, s} D \iso B \tensor_{C,t} D$, respectively, satisfying the cocycle condition.  We have a commutative diagram 
$$\xymatrix{
A^R \ar@{^(->}[r] \ar[d]^{\alpha}	& A \ar@<.5ex>[r]^{\beta_A \circ (\id \tensor 1)} \ar@<-.5ex>[r]_{\id \tensor 1} \ar[d]	& A \tensor_{C,t} D \ar[d] \\
B^R \ar@{^(->}[r] 		& B \ar@<.5ex>[r]^{\beta_B \circ (\id \tensor 1)} \ar@<-.5ex>[r] _{\id \tensor 1} 		&  B \tensor_{C,t} D
}$$ 
of exact sequences with $A^R = \Gamma(\cX, \cA)$ and $B^R = \Gamma(\cX, \cB)$. 

\medskip \noindent
Let $b \in B^R$ and choose $a \in A$ with $a \mapsto b$.  Then multiplication by $\beta_A(a \tensor 1) \in A \tensor_{C,t} D$ is an $A$-module homomorphism (via $\id \tensor 1: A \to A \tensor_{C,t} D$).  The characteristic polynomial is
$$P(\lambda, \beta_A(a \tensor 1)) = \lambda^N - \sigma_{N-1} \lambda^{N-1} + \cdots + (-1)^N \sigma_0  \in A^R[\lambda]$$
which maps under $\alpha$ to the characteristic polynomial of $\beta_B(b \tensor 1) = b \tensor 1$
$$P(\lambda, \beta_B(b \tensor 1)) = (\lambda - b)^N$$

\medskip \noindent
By examining the constant term, we see that $\sigma_0 \in A^R$ and $\sigma_0 \mapsto b^N$.
\end{proof}

\subsection{Equivalences}

\begin{lem} \label{lemma-finite}
Suppose that $\cX$ and $\cX'$ are algebraic stacks and that
$$\xymatrix{
\cX' \ar[r]^f \ar[d]^{\phi'}		& \cX \ar[d]^{\phi} \\
Y' \ar[r]^{g}					& Y
}$$
is a commutative diagram with $\phi$ and $\phi'$ adequate moduli spaces.  Assume that the following hold:
\begin{enumeratea}
\item The morphism $f$ is representable, quasi-finite and separated.
\item The morphism $g$ is integral.
\item The morphism $f$ maps closed points to closed points.
\end{enumeratea}
Then $f$ is finite.
\end{lem}

\begin{proof} It suffices to show that $\cX' \to \cX \times_{Y} Y'$ is integral so we may 
assume that $Y = Y'$.  By Zariski's Main Theorem (\cite[Theorem 16.5]{lmb}), there
 exists a factorization $f: \cX'  \hookarr \tilde \cX \to \cX$ with $i: \cX' \hookarr \cX$ an 
 open immersion and $\tilde \phi: \tilde \cX \to \cX$ integral.  Since $f$ maps closed 
 points to closed points, so does $i$.  It follows from 
 Lemma \ref{lemma-adequate-affine-algebra} that $\tilde \cX \to Y$ is an adequate 
 moduli space.  If $x \in |\tilde \cX| \setminus |\cX'|$ is a closed point, then 
 $\tilde \phi(x) \in Y$ is closed.  Let $x'$ be the unique closed point in the
  fiber $\phi'^{-1}(y)$.  Then $i(x') \in |\tilde \cX|$ is the unique closed point in the fiber 
  $\tilde \phi^{-1}(y)$ by Theorem \ref{theorem-adequate}(\ref{points}) so
  $i(x') = x$.  It follows that $\cX' = \tilde \cX$ and $\cX' \to \cX$ is integral.
\end{proof}

\begin{thm} \label{theorem-coarse}
 Suppose that $\cX$ is an algebraic stack
with quasi-finite and separated diagonal.  Then the following are equivalent:
\begin{enumerate1}
\item The inertia $I_{\cX} \to \cX$ is finite.
\item There exists a coarse moduli space $\phi: \cX \to Y$ with $\phi$ separated.
\item There exists an adequate moduli space $\phi: \cX \to Y$.
\end{enumerate1}
\end{thm}

\begin{proof}  The Keel$-$Mori theorem (see Theorem \ref{theorem-keel-mori}) shows that $(1) \iff (2)$.  Proposition \ref{proposition-coarse} shows that $(2) \implies (3)$.  Suppose that statement (3) holds.  We may assume that $Y$ is separated.  First note that $\cX \to \cX \times \cX$ maps closed points to closed points.  Since $\phi \times \phi: \cX \times \cX \to Y \times Y$ is adequately affine, there exists a diagram
$$\xymatrix{
\cX \ar[r] \ar[d]^{\phi} 		& \cX \times \cX \ar[d]^{\varphi} \\
Y \ar[r]				& Z
}$$
where $\varphi: \cX \times \cX \to Z := \sSpec_{Y \times Y} (\phi \times \phi)_* \oh_{\cX \times \cX}$ is an adequate moduli space and $Y \to Z$ is integral (by Proposition \ref{proposition-base-change}(3)). It follows from Lemma \ref{lemma-finite}
that $\cX \to \cX \times \cX$ is finite.
\end{proof}

\begin{example} Let $X$ be the bug-eyed cover of the affine line over a field $k$ with $\text{char}(k) \neq 2$; that is, $X$ is defined by the quotient of the \'etale equivalence relation
$$\ZZ / 2 \ZZ \times \AA^1 \setminus \{(-1, 0)\} \rrarrows \AA^1 = \Spec(k[x])$$
where $\ZZ / 2 \ZZ$ acts via $x \mapsto -x$.  Then $X \to \AA^1 = \Spec(k[x^2])$ is a universal homeomorphism such that $\Gamma(X, \oh_X) = k[x^2]$.  However, $X \to \AA^1$ is not an adequate moduli space.  If $\text{char}(k) = 0$, then taking global sections of the surjection $\oh_X \to \oh_X / I^2$ where $I$ is ideal sheaf of the origin yields $k[x^2] \to k[x]/x^2$ which is not adequate.  If $\text{char}(k) = p \neq 2$, then consider the quasi-coherent $\oh_X$-algebra $\oh_X[t]$ where the action is given by $t \mapsto -t$ 
Then taking global sections of the surjection $\oh_X[t] \to \oh_X[t] / I^2 \oh_X[t]$ yields $k[x^2, t^2] \to k[x,t]/x^2$.  But there is no power of $x+t \in k[x,t]/x^2$ which is in the image.
\end{example}

%%%%%
\section{Geometrically reductive group schemes and GIT} \label{section-adequate-groups}

\medskip \noindent
In this section we introduce the notion of a geometrically reductive group algebraic space $G \to S$ over an arbitrary algebraic space $S$.  Our notion is equivalent to Seshadri's notion in \cite{seshadri_reductivity} (see Lemma \ref{lemma-group-equiv-affine} and Remark \ref{remark-seshadri}) when $G \to \Spec(R)$ is a flat, separated, finite type group scheme over a Noetherian affine scheme which satisfies the resolution property.  

\medskip \noindent The following are the main examples of geometrically reductive group algebraic spaces.
\begin{enumerate}
\item Any linearly reductive group algebraic space is geometrically reductive (see Remark \ref{remark-linearly-reductive}).
\item Any flat, finite, finitely presented group algebraic space is geometrically reductive (see Theorem \ref{theorem-finite-group}).  In particular, any finite group is geometrically reductive.
\item Any smooth affine group scheme $G \to S$ such that $G^{\circ} \to S$ is reductive and $G/G^{\circ} \to S$ is finite is geometrically reductive (see Theorem \ref{theorem-reductive}).  In particular, any reductive group scheme (for example, $\GL_{n,S} \to S$, $\PGL_{n,S} \to S$ or $\SL_{n,S} \to S$) is geometrically reductive.  
\end{enumerate}

\subsection{Definition and GIT}
\begin{defn} \label{definition-geometrically-reductive}
Let $S$ be an algebraic space.  A flat, separated, finitely presented group algebraic space $G \to S$ is \emph{geometrically reductive} if the morphism $BG \to S$ is adequately affine.
\end{defn}

\begin{remark} In other words, this definition is requiring that for every surjection $\cA \to \cB$ of quasi-coherent $G$-$\oh_S$-algebras, $\cA^G \to \cB^G$ is adequate.
\end{remark}

\begin{remark} \label{remark-linearly-reductive}
  This notion is weaker than the notion of linearly reductivity introduced in \cite[Section 12]{alper-good}.  Recall that a flat, separated, finitely presented group algebraic space $G \to S$ is \emph{linearly reductive} if $BG \to S$ is cohomologically affine; that is, if the functor from quasi-coherent $G$-$\oh_S$-modules to quasi-coherent $\oh_S$-modules given by taking invariants 
$$\text{QCoh}^G(S) \to \text{QCoh}^G, \quad \cF \mapsto \cF^G$$
is exact. If $S$ is defined over $\Spec(\QQ)$, then it follows from Lemma \ref{lemma-adequately-affine-char0} that $G \to S$ is linearly reductive if and only if $G \to S$ is geometrically reductive.   We emphasize that it is in characteristic $p$ where the notions are not equivalent.  The group schemes $\ZZ / p \ZZ$ and  $\GL_2$ are geometrically reductive but not linearly reductive.  In fact, an algebraic group $G$ over an algebraically closed field of characteristic $p$ is linearly reductive if and only if the connected component of the identity $G^{\circ}$ is a torus and $G/G^{\circ}$ has order prime to $p$ \cite{nagata_complete}.  
\end{remark}

\begin{thm} \label{theorem-git}
Let $S$ be an algebraic space.  Let $G \to S$ be a geometrically reductive group algebraic space acting on an algebraic space $X$ with $p: X \to S$ affine.  Then
$$\phi: [X/G] \to \sSpec_S (p_* \oh_X)^G$$
is an adequate moduli space.
\end{thm}

\begin{proof}
 Since $[X/G] \to BG$ is affine, the composition $[X/G] \to BG \to S$ is adequately affine so the statement follows.
 \end{proof}

\begin{remark}  With the notation of Theorem \ref{theorem-git}, if $S$ is affine and $X = \Spec(A)$, then the theorem implies that
$$\phi: [\Spec(A) /G] \to \Spec(A^G)$$
is an adequate moduli space.
\end{remark}

\subsection{Equivalences} We will give equivalent definitions of adequacy first in the general case $G \to S$ of a group algebraic space, then in the case where $S$ is affine, and finally in the case where $S$ is the spectrum of a field.  
We call a morphism $\cA \to \cB$ of quasi-coherent $G$-$\oh_S$-algebras \emph{universally adequate} if the corresponding morphism of $\oh_{BG}$-algebras is.

\begin{lem} \label{lemma-group-equiv}
Let $S$ be an algebraic space.  Let $G \to S$ be a flat, separated, finitely presented group algebraic space.
\begin{enumerate1}
\item For every universally adequate morphism $\cA \to \cB$ of $G$-$\oh_S$-algebras with kernel $\cI$, the morphism $\cA^G / \cI^G \to \cB^G$ is an adequate homeomorphism.
\item The morphism $G \to S$ is geometrically reductive.
\item For every surjection $\cF \to \cG$ of $G$-$\oh_S$-modules, the morphism $(\Sym^* \cF)^G \to (\Sym^* \cG)^G$ is universally adequate.
\end{enumerate1}
\medskip \noindent
If in addition $S$ is Noetherian, then the above are equivalent to:
\begin{enumerate1} 
\item[$(1')$] For every universally adequate morphism $\cA \to \cB$ of finite type $G$-$\oh_S$-algebras with kernel $\cI$, the morphism $\cA^G / \cI^G \to \cB^G$ is an adequate homeomorphism.
\item[$(2')$] For every surjection $\cA \to \cB$ of finite type $G$-$\oh_S$-algebras, the morphism $\cA^G \to \cB^G$ is universally adequate.  
\item[$(3')$] For every surjection $\cF \to \cG$ of finite type $G$-$\oh_S$-modules,  the morphism $(\Sym^* \cF)^G \to (\Sym^* \cG)^G$ is universally adequate.
\end{enumerate1}
\end{lem}

\begin{proof} This follows from Lemma \ref{lemma-adequately-affine-equiv}.  
\end{proof}

\medskip \noindent
We recall the following notion.
 
\begin{defn} A flat, separated, finitely presented group algebraic space $G \to S$ satisfies the \emph{resolution property} if for every finite type $G$-$\oh_S$-module $\cF$, there exists a surjection $\cV \to \cF$ from a locally free $G$-$\oh_{S}$-module $\cV$ of finite rank. 
\end{defn}

\begin{remark}  If $S = \Spec(R)$ is affine, then this is equivalent to requiring that for every finite type $G$-$R$-module $M$, there exists a surjection $V \to M$ of $G$-$R$-modules from a \emph{free} finite type $G$-$R$-module $V$.  Indeed, suppose that $V \to M$ is a surjection of $G$-$R$-modules where $V$ is a locally free $G$-$R$-module.  We may choose a surjection $R^{\oplus n} \to V$ as $R$-modules which then splits as $R^{\oplus n} = V \oplus V'$.  If we give $V'$ the trivial $G$-$R$-module structure, we see that we have a surjection $R^{\oplus n} \to V \to M$ of $G$-$R$-modules.
\end{remark}

\begin{remark}  \label{remark-resolution}
In \cite{thomason}, Thomason shows that a group scheme $G \to S$ satisfies the resolution property in the following cases: % CHECK THE PAPER
\begin{enumerate}
\item The scheme $S$ is a separated, regular Noetherian scheme of dimension $\le 1$ and $G \to S$ is affine.
\item The scheme $S$ is a separated, regular Noetherian scheme of dimension $\le 2$ and $\pi: G \to S$ is affine, flat and finitely presented such that $\pi_* \oh_G$ is locally a projective module over $\oh_S$ (for example, $G \to S$ is smooth with connected fibers).
\item The scheme $S$ has an ample family of line bundles (e.g., $S$ is regular or affine), $G \to S$ is a reductive group scheme such that either (i) $G$ is split reductive, (ii) $G$ is semisimple, (iii) $G$ has isotrivial radical and coradical or (iv) $S$ is normal.
\end{enumerate}
\end{remark}

\begin{lem} \label{lemma-group-equiv-affine}
 Let $G \to \Spec(R)$ be a flat, separated, finitely presented group algebraic space.  The following are equivalent:
\begin{enumerate1}
\item For every universally adequate morphism $A \to B$ of $G$-$R$-algebras with kernel $K$, the induced $R$-algebra homomorphism $A^G / K^G \to B^G$ is an adequate homeomorphism.
\item The morphism $G \to S$ is geometrically reductive.
\item For every surjection $A \to B$ of $G$-$R$-algebras, the homomorphism $A^G \to B^G$ is adequate. 
\item For every surjection $M \to N$ of $G$-$R$-modules, the homomorphism $(\Sym^* M)^G \to (\Sym^* N)^G$ is adequate.
\item For every surjection $M \to R$ of $G$-$R$-modules where $R$ has the trivial $G$-$R$-module structure, there exist $N > 0$ and $f \in (\Sym^N M)^G$ such that $f \mapsto 1$ under $(\Sym^N M)^G \to R$ is adequate.
\end{enumerate1}
If in addition $R$ is Noetherian, then the above are equivalent to:
\begin{enumerate1}  
\item[$(1')$] For every universally adequate morphism $A \to B$ of finite type $G$-$R$-algebras with kernel $K$, the induced $R$-algebra homomorphism $A^G / K^G \to B^G$ is an adequate homeomorphism.
\item[$(2')$] For every surjection $A \to B$ of finite type $G$-$R$-algebras, the homomorphism $A^G \to B^G$ is universally adequate.
\item[$(3')$] For every surjection $A \to B$ of finite type $G$-$R$-algebras, the homomorphism $A^G \to B^G$ is adequate.
\item[$(4')$]  For every surjection $M \to N$ of finite type $G$-$R$-modules, the homomorphism $(\Sym^* M)^G \to (\Sym^* N)^G$ is adequate.
\item[$(5')$] For every surjection $M \to R$ of finite type $G$-$R$-modules where $R$ has the trivial $G$-$R$-module structure, there exist $N > 0$ and $f \in (\Sym^N M)^G$ such that $f \mapsto 1$ under $(\Sym^N M)^G \to R$ is adequate.
\item[$(6')$] For every finite type free $G$-$R$-module $V$, $R$-algebra $k$ with $k$ a field and non-zero homomorphism of $G$-$R$-modules $V \to k$, there exist $n > 0$ such that  $(\Sym^n V)^G \to k$ is non-zero.
\end{enumerate1}
If in addition $G \to \Spec(R)$ satisfies the resolution property, then the above are equivalent to
\begin{enumerate1}
\item[$(1'')$] For every universally adequate morphism $R[x_1, \cdots, x_n] \to B$ of finite type $G$-$R$-algebras with kernel $K$, the induced $R$-algebra homomorphism $R[x_1, \cdots, x_n]^G / K^G \to B^G$ is an adequate homeomorphism.
\item[$(2'')$] For every surjection $R[x_1, \cdots, x_n] \to B$ of finite type $G$-$R$-algebras,  the homomorphism $R[x_1, \cdots, x_n]^G \to B^G$ is universally adequate.
\item[$(3'')$] For every surjection $R[x_1, \cdots, x_n] \to B$ of finite type $G$-$R$-algebras,  the homomorphism $R[x_1, \cdots, x_n]^G \to B^G$ is adequate.
\item[$(4'')$]  For every surjection $V \to N$ of finite type $G$-$R$-modules where $V$ is free as an $R$-module, the homomorphism $(\Sym^* M)^G \to (\Sym^* N)^G$ is adequate.
\item[$(5'')$] For every surjection $V \to R$ of finite type $G$-$R$-modules where $V$ is free as an $R$-module and $R$ has the trivial $G$-$R$-module structure, there exist $N > 0$ and $f \in (\Sym^N V)^G$ such that $f \mapsto 1$ under $(\Sym^N V)^G \to R$ is adequate.
\item[$(6'')$] For every finite type $G$-$R$-module $V$ which is free as an $R$-module, $R$-algebra $k$ with $k$ a field and non-zero homomorphism of $G$-$R$-modules $V \to k$, there exist $n > 0$ such that $(\Sym^n V)^G \to k$ is non-zero.
\item[$(7'')$] For every finite type $G$-$R$-module $V$ which is free as an $R$-module and invariant vector $v \in V^G$ such that $R \xrightarrow{v} V$ is injective, there exists a non-zero homogenous invariant polynomial $f \in (\Sym^n V^{\dual})^G$ with $f(v) = 1$.
\end{enumerate1}
\end{lem}

\begin{proof} The equivalences (1)$-$(5), (1')$-$(5') and (1'')$-$(5'') follow from Lemma \ref{lemma-affine-base-equiv}.  It is clear that that $(5') \implies (6')$.  Conversely, suppose that $M \to R$ is a surjection of finite type $G$-$R$-modules.  Let $Q$ be the cokernel of the induced map $\alpha: (\Sym^* M)^G \to (\Sym^* R) \cong R[x]$.  We need to show that there exists $N > 0$ such that the image of $x^N$ in $Q$ is 0.  For every $\fp \in \Spec(R)$, we know there exist $n > 0$ and $f \in (\Sym^n V)^G$ such that $f(p) \neq 0$; that is, $\alpha(f) = c x^n$ with $c \in R \setminus \fp$.  Therefore, for every $\fp \in \Spec(R)$, there exists $n> 0$ such that $x^n$ is non-zero in $Q_{\fp}$.  Since $R$ is Noetherian, there exists $N > 0$ such that $x^N = 0 \in Q$.  The equivalences  $(5'') \iff (6'') \iff (7'')$ are similar.
\end{proof}

\begin{remark} \label{remark-seshadri}
Property $(6'')$ translates into the geometric condition that for every linear action of $G$ on $X = \AA^n_R = \Spec(\Sym V^{\dual})$ over $R$, where $V$ is a finite type $G$-$R$-module which is free as an $R$-module, and for every field-valued point $x_0 \in X(k) = V \tensor_R k$ which is $G \times_R k$-invariant, there exist $n > 0$ and a $G$-invariant element $f \in (\Sym^n V^{\dual})^G$ such that $f(x_0) \neq 0$. This is precisely Seshadri's condition of geometric reductivity in \cite[Theorem 1]{seshadri_reductivity} (see also \cite[Appendix G to Ch. 1]{git3}).
\end{remark}

\begin{remark}  \label{remark-dual}
Property $(5'')$ translates into the geometric condition that for every linear action of $G$ on $X = \AA^n_R = \Spec(\Sym V^{\dual})$ over $R$, where $V$ is a finite type $G$-$R$-module which is free as an $R$-module, and for every $G$-invariant $x \in X(R)$ which is given by an inclusion $R \hookarr V$ of $G$-$R$-modules, there exists $f \in (\Sym^n V^{\dual})^G$ such that $f(x) =1$.
\end{remark}

\begin{lem}  \label{lemma-group-equiv-field}
Let $G \to \Spec(k)$ be a finite type group scheme where $k$ is a field.  The following are equivalent:
\begin{enumerate1}
\item The group scheme $G$ is geometrically reductive.
\item For every surjection $V \to W$ of $G$-representations and $w \in W^G$, there exist $N > 0$ and $v \in (\Sym^N V)^G$ with $v \mapsto w^N$.
\item For every linear action of $G$ on $\AA^n$, closed invariant subscheme $Z$ and $G$-invariant function $f$ on $Z$ , there exists $N>0$ such that $f^N$ extends to a $G$-invariant function on $X$. 
\item For every linear action of $G$ on $\AA^n$ and closed invariant $k$-valued point $q \in \AA^n$, there exists a homogenous invariant non-constant polynomial $f$ such that $f(q) \ne 0$. 
\item For every $G$-representation $V$ and codimension one invariant subspace $W \subseteq V$, there exist $r > 0$ and a dimension one invariant subspace $Q \subseteq \Sym^{p^r}V$ such that $\Sym^{p^r} V \cong (W \cdot \Sym^{p^r-1} V) \oplus Q$.
\end{enumerate1}
If $G \to \Spec(k)$ is affine and smooth, then (1) $-$ (5) are equivalent to:
\begin{enumerate1}
\item[(6)] The group scheme $G$ is reductive.
\item[(7)] For every action of $G$ on a finite type $k$-scheme $\Spec(A)$, the ring of invariants $R^G$ is finitely generated.
\end{enumerate1}
If $\text{char}(k) = 0$, then (1) $-$ (5) are equivalent to:
\begin{enumerate1}
\item[(8)] The group scheme $G$ is linearly reductive.
\end{enumerate1}
If $\text{char}(k) = p$, then (1) - (5) are equivalent to:
\begin{enumerate1}
\item[(9)] For every linear action of $G$ on $\AA^n$ and closed invariant $k$-valued point $q \in \AA^n$, there exists a homogenous invariant polynomial $f$ of degree $p^r$ for some $r > 0$ such that $f(q) \ne 0$. 
\end{enumerate1}
\end{lem}

\begin{proof}
The equivalence $(1) \iff (2) \iff (3) \iff (4)$ is provided by Lemma \ref{lemma-group-equiv-affine}.  Statement (5) was Mumford's original formulation of geometric reductivity in \cite[Preface]{git} and is easily seen to be equivalent to the others.  The equivalence $(1) \iff (6)$ is Haboush's theorem \cite{haboush}.  The equivalence $(6) \iff (7)$ is provided by \cite[Lemma A.1.2]{git3}. The equivalence of $(1) \iff (8)$ follows from Lemma \ref{lemma-adequately-affine-char0} and $(1) \iff (9)$ follows from Lemma \ref{lemma-adequate-charp}.
\end{proof}

%%%%%%%%%%%%%%%%%%%%%%%
\subsection{Base change, descent and stabilizers}

\begin{prop} \label{proposition-group-base-change}
Let $S$ be an algebraic space and let $G \to S$ be a flat, finitely presented and separated group algebraic space.  Let $S' \to S$ be a morphism of algebraic spaces.
\begin{enumeratei}
\item If $G \to S$ is geometrically reductive, so is $G \times_S S' \to S'$.
\item If $S' \to S$ is faithfully flat and $G \times_S S' \to S'$ is geometrically reductive, then so is $G \to S$.
\end{enumeratei}
\end{prop}

\begin{proof} Since $BG' = BG \times_S S'$, this follows directly from Proposition \ref{proposition-adequately-affine}. 
\end{proof}

\medskip \noindent
The following definition is justified by fpqc descent in Proposition \ref{proposition-adequately-affine}(4). 
\begin{defn}  If $\cX$ is an algebraic stack, a point $\xi \in |\cX|$ has a \emph{geometrically reductive stabilizer} if for some (equivalently any) representative $x: \Spec(k) \to \cX$, the stabilizer group scheme $G_x \to \Spec(k)$ is geometrically reductive.
\end{defn}

\begin{remark} If $\cX$ is locally Noetherian, then there exists a residual gerbe $\cG_{\xi} \subseteq \cX$ and $\xi \in |\cX|$ has geometrically reductive stabilizer if and only if $\cG_{\xi}$ is adequately affine.
\end{remark}

\medskip \noindent
The following is an easy but useful fact ensuring that closed points have geometrically reductive stabilizers.

\begin{prop} \label{closed_orbits_red}
Let $\cX$ be a locally Noetherian algebraic stack and let $\phi: \cX \to Y$ be an adequate moduli space.  Any closed point $\xi \in |\cX|$ has a geometrically reductive stabilizer.  For any $y \in Y$, the unique closed point $\xi \in |\cX_y|$ has a geometrically reductive stabilizer .
\end{prop}

\begin{proof} The point $\xi$ induces a closed immersion $\cG_{\xi} \hookarr \cX$.  By Lemma \ref{lemma-adequate-affine-algebra}, the morphism from $\cG_{\xi} \to \Spec(k(\xi))$ is an adequate moduli space so that $\xi$ has geometrically reductive stabilizer.
\end{proof}

\subsection{Matsushima's theorem} \label{section-matsushima}
\medskip \noindent 
In \cite[Theorem 3]{matsushima}, Matsushima proved using analytic methods that if $G$ is a semisimple complex Lie group and $H \subseteq G$ is a closed, connected complex subgroup, then $H$ is reductive if and only if $G/H$ is affine.  
Bialynicki-Birula gave an algebro-geometric proof in \cite{bb_homogeneous} using a result from \cite{bbhm} that if $G$ is a reductive group over a field of characteristic $0$ and $H \subseteq G$ is a closed subgroup, then $H$ is reductive if and only if $G/H$ is affine.  It was known that the transcendental proof given in \cite[Theorem 3.5]{borel-harish-chandra} works in arbitrary characteristic but it relied on sophisticated \'etale cohomology methods.
Richardson gave a direct proof in \cite{richardson} that this holds for arbitrary algebraically closed fields $k$ using Haboush's theorem equating reductive groups and geometrically reductive groups.   
Haboush establishes in \cite[Proposition 3.2]{haboush_stab} that if $G$ is a geometrically reductive linear algebraic group over any field $k$ and $H \subseteq G$ is a closed subgroup, then $H$ is geometrically reductive if and only if $G/H$ is affine; from Haboush's theorem, he therefore deduces the analogous statement for reductive groups. There is also a proof by Ferrer Santos in \cite{ferrer-santos_reductive}, based on the techniques in \cite{cline-parshall-scott}, of the statement for geometrically reductive groups over an algebraically closed field.

\medskip \noindent We now give a generalization of Matsushima's theorem.  See also Corollary \ref{corollary-matsushima}.

\begin{thm} \label{theorem-matsushima} 
Suppose that $S$ is an algebraic space. Let $G \to S$ be a geometrically reductive group algebraic space and let $H \subseteq G$ be a flat, finitely presented and separated subgroup algebraic space.  If $G/H \to S$ is affine, then $H \to S$ is geometrically reductive.  If $G \to S$ is affine, the converse is true.
\end{thm}

\begin{proof}
Consider the cartesian diagram
$$\xymatrix{
G/H \ar[r] \ar[d]	 \ar @{} [dr] |{\square}		& S \ar[d] \\
BH \ar[r]			& BG
}$$
If $G/H \to S$ is affine, then by descent $BH \to BG$ is affine.  Therefore, the composition $BH \to BG \to S$ is adequately affine, so $H \to S$ is geometrically reductive.  Conversely, if $G \to S$ is affine and $H \to S$ is geometrically reductive, then $G/H \to BH$ is affine and the composition $G/H \to BH \to S$ is adequately affine.  It follows from the generalization of Serre's criterion (Theorem \ref{serres-criterion}) that $G/H \to S$ is affine.
\end{proof}

\begin{cor} Let $S$ be an algebraic space.  Suppose that $G \to S$ is a geometrically reductive group algebraic space acting on an algebraic space $X$ affine over $S$.  Let $x: \Spec(k) \to X$.  If the orbit $o(x) \subseteq X \times_S k$ is affine, then $G_x \to \Spec(k)$ is geometrically reductive.  Conversely, if $G \to S$ is affine and $G_x \to \Spec(k)$ is geometrically reductive, then the orbit $o(x)$ is affine. \epf
\end{cor}

\subsection{Quotients and extensions}
\begin{prop} \label{proposition-extensions}
Consider an exact sequence of flat, finitely presented and separated group algebraic spaces
$$1 \to G' \to G \to G'' \to 1 $$
\begin{enumerate}  
\item If $G \to S$ is geometrically reductive, then $G'' \to S$ is geometrically reductive.
\item If $G' \to S$ and $G'' \to S$ are geometrically reductive, so is $G \to S$.
\end{enumerate}
\end{prop}

\begin{proof}
Consider the 2-commutative diagram
$$
\xymatrix{
BG' \ar[r]^i \ar[rd]_{\pi_{G'}	}	& BG \ar[r]^j \ar[d]^{\pi_G} 	& BG'' \ar[ld]^{\pi_{G''} }\\
						& S
}
\qquad 
\xymatrix{
BG' \ar[r]^i \ar[d]^{\pi_{G'}}	 \ar @{} [dr] |{\square}		& BG \ar[d]^j \\
S \ar[r]^p					&BG''
}
$$
where the right square is cartesian and the functors $i^*$ and $j^*$ are exact (on quasi-coherent sheaves).  The natural adjunction morphism $\id \to j_* j^*$ is an isomorphism; indeed it suffices to check that $p^* \to p^* j_* j^*$ is an isomorphism and there are canonical isomorphisms $p^* j_* j^* \cong {\pi_{G'}}_*i^* j^* \cong {\pi_{G'}}_* \pi_{G'}^* p^*$ such that the composition $p^* \to {\pi_{G'}}_* \pi_{G'}^* p^*$ corresponds the composition of $p^*$ and the adjunction isomorphism $\id \to {\pi_{G'}}_* \pi_{G'}^*$.

\medskip \noindent
To prove part (1), we have isomorphisms of functors
$${\pi_{G''}}_* \iso {\pi_{G''}}_* j_* j^* \cong {\pi_G}_* j^*$$
and since $j^*$ is exact and $\pi_G$ is adequately affine, $\pi_{G''}$ is adequately affine.

\medskip \noindent
To prove part (2), $j$ is adequately affine since $p$ is faithfully flat and $G' \to S$ is geometrically reductive.   As $\pi_G = \pi_{G''} \circ j$ is the composition of adequately affine morphisms, $G \to S$ is geometrically reductive.
\end{proof}

\subsection{Flat, finite, finitely presented group schemes are geometrically reductive}

\medskip \noindent
It was shown in \cite[Theorem 1]{waterhouse} than any finite group scheme $G$ (possibly non-reduced) over a field  $k$ is geometrically reductive.  We show that this holds over an arbitrary base: 

\begin{thm} \label{theorem-finite-group} 
Let $S$ be an algebraic space and let $G \to S$ be a quasi-finite, separated, flat group algebraic space.  Then $G \to S$ is geometrically reductive if and only if $G \to S$ is finite.
\end{thm}

\begin{proof}
This follows directly from Theorem \ref{theorem-coarse}.
\end{proof}

\begin{example}
Let $k$ be a field and let $G \to \AA^1 = \Spec(k[x])$ be the group scheme with fibers isomorphic to $\ZZ_2$ everywhere except over the origin where it is trivial.  
It follows from Theorem \ref{theorem-finite-group} that since $G \to \AA^1$ is quasi-finite but not finite, $G \to \AA^1$ is not adequate.  One can also see this directly.  Suppose $\text{char}(k) \neq 2$.  Consider the action of $G$ on $X = \Spec(k[x,y])$ over $\AA^1$ defined by the involution $\sigma: k[x,y]_x \to k[x,y]_x$ given by $\sigma(y) = - y$.  Then if $\cX = [X/G]$ and $\cZ$ is the closed substack defined by $x = 0$, then 
$$k[x,y^2] \cong \Gamma(\cX, \oh_{\cX}) \to \Gamma(\cX, \oh_{\cZ}) \cong k[y]$$
is not adequate as there is no prime power of $y+1$ which lifts.  (One can show in a similar way that $G \to \AA^1$ is not adequate if $\text{char}(k) = 2$.).  
\end{example}

%%%%
\subsection{Reductive group schemes are geometrically reductive} \label{subsection-reductive}

Following Seshadri \cite{seshadri_reductivity}, we generalize Haboush's theorem \cite{haboush} to show that reductive group schemes are geometrically reductive.  Seshadri's result \cite[Theorem 1]{seshadri_reductivity} states that any reductive group scheme $G \to \Spec(R)$ with $R$ Noetherian satisfies property \ref{lemma-group-equiv-affine}$(6'')$.  We show that Seshadri's method generalizes to establish that a reductive group scheme is geometrically reductive according to Definition \ref{definition-geometrically-reductive}.  We stress that this is only a mild generalization of \cite[Theorem 1]{seshadri_reductivity} as our notion is equivalent to Seshadri's notion for flat, finite type, separated group schemes $G \to S$ that satisfy the resolution property with $S$ Noetherian and affine .

\medskip \noindent
The only improvement in our proof is that systematically developing the theory of geometrically reductive group schemes (for example, properties of base change, descent and extensions) simplifies the reductions to the case where G is a semisimple group scheme over a discrete valuation ring (DVR) with algebraically closed residue field.  However, the heart of the argument is in the representation theory in \cite[Property I and II on pg. 247]{seshadri_reductivity}  (see Lemmas \ref{lemma-propertyI} and \ref{lemma-propertyII}).

\begin{defn} A group scheme $G \to S$ is \emph{reductive} if $G \to S$ is affine and smooth such that the geometric fibers are connected and reductive.
 \end{defn}

\medskip \noindent
Let $G \to \Spec(R)$ be split reductive group scheme (\cite[Exp. XXII, Definition 1.13]{sga3}).  Fix a split maximal torus $T \subseteq G$ and a Borel subgroup scheme $B \supseteq T$.  Let $U \subseteq B$ be the unipotent subgroup scheme.  Denote by $X(T)$ the group of characters $T \to \GG_m$.  Let $\rho \in X(T)$ be the half sum of positive roots.  Then $\rho$ extends to a homomorphism $\tilde \rho: B \to \GG_m$ defined functorially by $\tilde \rho (tu) = \rho(t)$ for $t \in T$ and $u \in U$.  For a positive integer $m$, define
$$ W_{m \rho} = \{f \in \Gamma(G, \oh_G) \, | \, f(gb) = \tilde \rho(b)^m f(g) \text{ for all } b \in B \}$$
If $L$ is the line bundle on $G/B$ associated with $\rho$, then one can identify $W_{m \rho}$ with the $R$-module of sections $\Gamma(G/B, L^m)$.

 \begin{lem} (\cite[Property I on pg. 247]{seshadri_reductivity})
  \label{lemma-propertyI}
  Let $G$ be a split semisimple and simply connected group scheme over a DVR $R$ with algebraically closed residue field $\kappa$.  Fix a maximal torus $T$ and a Borel $B$ containing it.   Then
  \begin{enumerate1}
  \item For $m > 0$,  $(W_{m \rho} \tensor_R W_{m \rho})^G \cong R$.
  \item If $V$ is a finite type free $G$-$R$-module and $v \in V^G$, for $m \gg 0$, there is a homomorphism of $G$-$R$-modules
$$\varphi: V \to W_{m \rho} \tensor_R W_{m \rho}$$
such that the image of $\varphi(v)$ in $W_{m \rho} \tensor_R W_{m \rho} \tensor_R \kappa$ is non-zero. 
  \end{enumerate1}
 \end{lem}
 
 \begin{remark}  Statement (1) above differs from \cite[Property I(a)]{seshadri_reductivity} which states that 
 $$((W_{m \rho} \tensor_R k) \tensor_k (W_{m \rho} \tensor_R k))^{G \times_R k} \cong k$$
  but follows in the same way from \cite[Lemma 3]{seshadri_reductivity}.
 \end{remark}
 
 \begin{lem} (\cite[Property II on pg. 247]{seshadri_reductivity})
  \label{lemma-propertyII}
  Let $G$ be a split semisimple and simply connected group scheme over a DVR $R$ with algebraically closed residue field $\kappa$.  Fix a maximal torus $T$ and a Borel $B$ containing it.   Then
 \begin{enumerate1} 
\item If $\text{char}(\kappa) = 0$, then for all $m > 0$  there is an isomorphism $W_{m \rho}^{\dual} \iso W_{m \rho}$.
\item If $\text{char}(\kappa) = p$, then for $m = p^{\nu} - 1$ with $\nu$ a positive integer there is an isomorphism $W_{m \rho}^{\dual} \iso W_{m \rho}$.
\end{enumerate1}
\end{lem}

\begin{thm} \label{theorem-haboush}
Let $G \to S$ be a smooth affine group scheme with connected fibers.  Then $G \to S$ is geometrically reductive if and only if $G \to S$ is reductive.
\end{thm}

 \begin{proof}
First, suppose that $G \to S$ is geometrically reductive.  By Proposition \ref{proposition-group-base-change}, for every $s: \Spec(k) \to \Spec(R)$, the base change $G_s \to \Spec(k)$ is a geometrically reductive, smooth and connected group scheme.  Let $R_u \subseteq G_s$ be the unipotent radical.  Since $G_s / R_u$ is an affine group scheme, 
Theorem \ref{theorem-matsushima}(1) shows that $R_u$ is geometrically reductive.  It follows that $R_u$ is trivial and that $G_s$ is reductive. 

\medskip \noindent
Now suppose that $G \to S$ is reductive. By \cite[Exp. XXII, Corollary 2.3]{sga3}, there exists an \'etale cover $S' \to S$ such that $G' =G \times_S S' \to S'$ is a split reductive group scheme.  By Proposition \ref{proposition-extensions}(1), it suffices to prove that $G' \to S'$ is geometrically reductive.  There exists a split reductive group scheme $H \to \Spec(\ZZ)$ such that $H \times_{\Spec(\ZZ)} S' \cong G'$.  By Theorem \ref{theorem-matsushima}(1), it suffices to prove that $H \to \Spec(\ZZ)$ is geometrically reductive.  Furthermore, by Proposition \ref{proposition-extensions}(1), we may assume that $G$ is a reductive group scheme over a DVR $R$ with algebraically closed residue field $\kappa$.
 
 \medskip \noindent
 The radical $R(G)$ of $G$ is a torus and thus geometrically reductive.  By Proposition \ref{proposition-extensions}(2), it suffices to show that $G/R(G)$ is geometrically reductive.  If $\tilde G \to G/R(G)$ is the simply connected covering of $G/R(G)$, then $\tilde G \to \ZZ$ is a split semisimple and simply connected group scheme.
 Furthermore, by Proposition \ref{proposition-extensions}(1), it suffices to show that $\tilde G$ is geometrically reductive.  Thus, we may assume that $G$ is a split semisimple and simply connected group scheme over a DVR $R$ with algebraically closed residue field $\kappa$.
 
 \medskip \noindent
Since $\dim R = 1$, $G$ satisfies the resolution property (see Remark \ref{remark-resolution}).  Using the equivalence of Lemma \ref{lemma-group-equiv-affine} and Remark \ref{remark-dual}, we need to show that given a finite type $G$-$R$-module $V$ which is free as an $R$-module and $x: R \to V$ an inclusion of $G$-$R$-modules, there exists $f \in (\Sym^n V^{\dual})^G$ such that $f(x) = 1$.
 By Lemma \ref{lemma-propertyI}(2) and Lemma \ref{lemma-propertyII} there exist $m>0$ and a homomorphism of $G$-$R$-modules
 $$\varphi: V \to W_{m \rho} \tensor_R W_{m \rho}^{\dual} \cong \Hom_R(W_{m\rho}, W_{m\rho})$$
 such that the image of $\varphi(v)$ in $W_{m \rho} \tensor_R W_{m \rho} \tensor_R \kappa$ is non-zero. 
Furthermore, by Lemma \ref{lemma-propertyI}(1), $\Hom_R(W_{m\rho}, W_{m\rho})^G$ is isomorphic to $R$ and is generated by the identity map $\id_{W_{m \rho}}: W_{m \rho} \to W_{m \rho}$.  It follows that $\varphi(v) = \lambda \cdot \id_{W_{m \rho}}$ where $\lambda \in R$ is a unit.  By multiplying $\varphi$ by $\lambda^{-1}$, we may assume $\varphi(v) = \id_{W_{m \rho}}$. The determinant function $\det: \Hom_R(W_{m\rho}, W_{m \rho}) \to R$ is a non-zero homogenous invariant polynomial.  Therefore the composition
 $$f: V \xrightarrow{\varphi} \Hom_R(W_{m\rho}, W_{m \rho}) \xrightarrow{\det} R$$
 is a non-zero homogenous invariant polynomial; that is, $f \in (\Sym^n V^{\dual})^G$ for some $n>0$.  Furthermore $f(v) = \det \varphi(v) = 1$ so we have constructed the desired invariant.
 \end{proof}
 
 \medskip \noindent
 If $G \to S$ is a smooth group scheme, then \cite[Exp. $VI_B$, Theorem 3.10]{sga3} implies that the functor 
  $$
  \begin{aligned}
  (\Sch/S) &\to \Sets \\
(T \to S)	& \mapsto  \{ g \in G(T) \, | \, \forall s \in S, \, g_s(T_s) \subseteq (G_s)^{\circ} \}
  \end{aligned}
  $$
is representable by an open subgroup scheme $G^{\circ} \subseteq G$ which is smooth over $S$.
 
 \begin{thm}  \label{theorem-reductive}
 Let $G \to S$ be a smooth affine group scheme such that $G/G^{\circ} \to S$ is separated.  Then $G \to S$ is geometrically reductive if and only if the geometric fibers of $G^{\circ} \to S$ are reductive and $G/G^{\circ} \to S$ is finite.
 \end{thm}
 
 \begin{proof}
 If $G \to S$ is geometrically reductive, then Proposition \ref{proposition-extensions} implies that $G/G^{\circ} \to S$ is geometrically reductive and Theorem \ref{theorem-finite-group} implies that $G/G^{\circ} \to S$ is finite.  Furthermore, the geometric fibers are geometrically reductive by Proposition \ref{proposition-group-base-change} and therefore reductive by Theorem \ref{theorem-haboush}.  Conversely, Theorem \ref{theorem-haboush} implies that $G^{\circ} \to S$ is geometrically reductive and Theorem \ref{theorem-finite-group} implies that $G/G^{\circ} \to S$ is geometrically reductive.  It follows from Proposition \ref{proposition-extensions}(2) that $G \to S$ is geometrically reductive.
 \end{proof}

\begin{cor} \label{corollary-matsushima}
If $G \to S$ is a reductive group scheme and $H \subseteq G$ is a smooth closed subgroup scheme with geometrically connected fibers, then $H \to S$ is reductive if and only if $G/H \to S$ is affine.
\end{cor}

\begin{proof}  This follows from Theorems \ref{theorem-matsushima} and \ref{theorem-haboush}.
\end{proof}

\medskip \noindent
We end with another immediate application of the theory of adequacy.  This result is well known to the experts but we are unaware of a reference.

\begin{prop}
Let $G \to S$ be a geometrically reductive group scheme (for example, a reductive group scheme).  Let $X \to Y$ be a morphism of algebraic spaces over $S$ which is a principal homogenous space for $G$.  If $X$ is affine over $S$, then so is $Y$.
\end{prop}

\begin{proof}
Let $p$ denote the structure morphism $X \to S$.  By Theorem \ref{theorem-git}, $[X/G]=Y \to \sSpec_S (p_* \oh_X)^G$ is an adequate moduli space.  In particular, $Y \to \sSpec_S (p_* \oh_X)^G$ is an adequately affine morphism of algebraic spaces so by the generalization of Serre's criterion (Theorem \ref{serres-criterion}), it is also affine; it follows that $Y$ is affine over $S$.
\end{proof}

\bibliography{refs}
\bibliographystyle{amsalpha}

\end{document}